\documentclass[12pt]{article}
\usepackage{verbatim}
\usepackage{latexsym}
\usepackage{amsmath,amsthm,amsfonts}
\usepackage{eucal}
\usepackage{cases}

\usepackage{mathtext}
\usepackage[cp1251]{inputenc}
\usepackage[T2A]{fontenc}
\usepackage{graphicx}
\usepackage{amsmath}
\usepackage{amssymb}
\usepackage{amsxtra}
\usepackage{latexsym}
\usepackage{ifthen}

\newtheorem{theorem}{Theorem}
\newtheorem{definition}{Definition}[section]
\newtheorem{lemma}{Lemma}[section]
\newtheorem{proposition}{Proposition}[section]

\numberwithin{equation}{section}

\renewcommand{\Re}{\mbox{Re}}
\newcommand{\Ga}{\mbox{Ga}}
\newcommand{\We}{\mbox{We}}
\newcommand{\R}{\mathbb{R}}

\newcommand{\eps}{\varepsilon}

\begin{document}

\title {\bf Nonnegative solutions for a long-wave unstable thin
film equation with convection}

\author{{\normalsize\bf Marina Chugunova, M. C. Pugh, R. M. Taranets}
\smallskip}

\date{\today}

\maketitle

\setcounter{section}{0}

\begin{abstract}
We consider a nonlinear 4th-order degenerate parabolic partial
differential equation that arises in modelling the dynamics of an
incompressible thin liquid film on the outer surface of a rotating
horizontal cylinder in the presence of gravity.  The parameters
involved determine a rich variety of qualitatively different flows.
Depending on the initial data and the parameter values, we prove the
existence of nonnegative periodic weak solutions. In addition, we
prove that these solutions and their gradients cannot grow any faster
than linearly in time; there cannot be a finite-time blow-up.
Finally, we present numerical simulations of solutions.
\end{abstract}


\textbf{keywords:} {fourth-order degenerate parabolic equations,
thin liquid films, convection, rimming flows, coating flows}

\section{Introduction} \label{A}
We consider the dynamics of a viscous incompressible fluid on the
outer surface of a horizontal circular cylinder that is rotating
around its axis in the presence of gravity, see Figure
\ref{cartoon}.
\begin{figure}
\begin{center}
\vspace{-.5in}
\includegraphics[height=5cm] {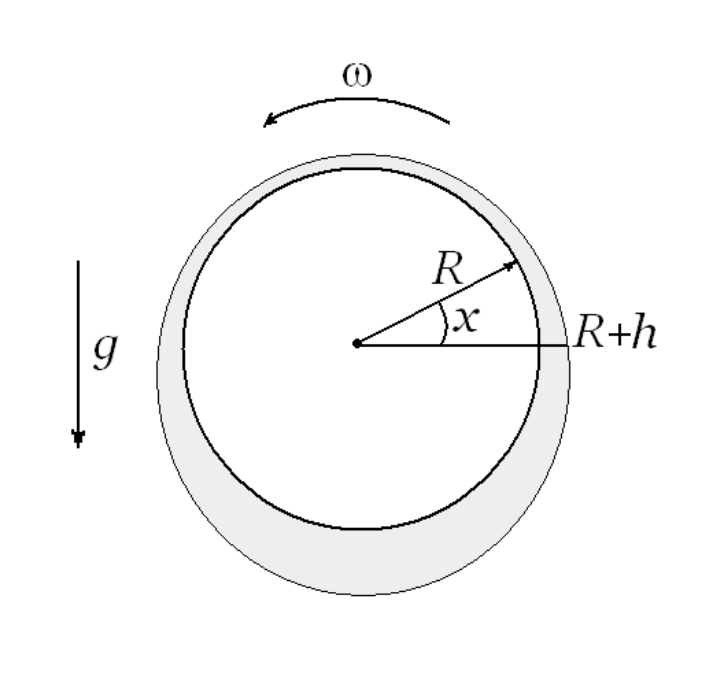}
\end{center}
\vspace{-.2in}
\caption{\label{cartoon}
Liquid film on the outer surface of a rotating
horizontal cylinder in the presence of gravity.
}
\end{figure}
If the cylinder is fully
coated there is only one free boundary: where the liquid meets the
surrounding air.  Otherwise, there
is also a free boundary (or contact line) where the air and liquid meet
the cylinder's surface.

The motion of the liquid film is governed by four physical effects:
viscosity, gravity, surface tension, and centrifugal forces. These
are reflected in the parameters: $R$ --- the radius of the cylinder,
$\omega$ --- its rate of rotation (assumed constant), $g$ --- the
acceleration due to gravity, $\nu$ --- the kinematic viscosity,
$\rho$ --- the fluid's density, and $\sigma$ --- the surface
tension.

These parameters yield three independent dimensionless numbers: the
Reynolds number $\Re = (R^2 \omega)/\nu$, the Galileo number $\Ga =
g/(R \omega^2)$ and the Weber number $\We = (\rho R^3
\omega^2)/\sigma$.

We introduce the parameter $\epsilon = \bar{h}/R$, where $\bar{h}$
is the average thickness of the liquid. The following quantities are
assumed to have finite, nonzero limits as $\epsilon \to 0$.
\cite{Pukh1, Pukh2, Benilov1, Moff}:

\begin{equation} \label{params}
\kappa = \Re\, \epsilon^2, \quad \chi =  \tfrac{\Re}{\We} \,
\epsilon^2, \quad \mbox{and} \quad  \mu = \Ga \, \Re \, \epsilon^2.
\end{equation}
This corresponds to a low rotation rate, for example.

One can model the flow using the full three-dimensional
Navier-Stokes equations with free boundaries: for $\vec{u}(x,y,z,t)$
in the region $x \in [-\pi,\pi)$, $y \in \R^1$, and $z \in
(0,h(x,y,t))$ where $x$ is the angular variable, $y$ is the axial
variable, and $h(x,y,t)$ is the thickness of the fluid above the
point $(x,y)$ on the surface of the cylinder at time $t$. This has
been done by Pukhnachov \cite{Pukh1} in which he considered the
physical regime for which the ratio of the free-fall acceleration
and the centripetal acceleration is small.  There, he proved the
existence and uniqueness of fully-coating steady states (no contact
line is present).  We know of no results for the affiliated initial
value problem.

In this physical regime, if one also makes a longwave approximation
(the thickness of the coating fluid is smaller than the radius of
the cylinder) and if one further assumes that the rotation rate is
low (or the viscosity is large) then the three-dimensional
Navier-Stokes equations with free boundary can be approximated by a
fourth-order degenerate partial differential equation (PDE) for the
film thickness $h(x,y,t)$. This is done by averaging  the fluid flow
in the direction normal to the cylinder \cite{Pukh1,Pukh2}. If one
further assumes that the flow is independent of the axial variable,
$y$, then this results in a PDE in one dimension for $h(x,t)$.

In his pioneering 1977 article about syrup rings on a rotating roller,
Moffatt neglected the effect of surface tension
(i.e. $\We^{-1}=0 = \chi$), assumed the flow was uniform in the
axial variable, and derived \cite{Moff} the following model for the
thin film thickness:
\begin{equation}
\label{A:MoffEq} h_t + \left( h - \tfrac{\mu}{3} h^3 \, \cos(x)
\right)_x = 0,
\end{equation}
where $\mu$ is given in (\ref{params}) and  $$ x \in [-\pi,\pi],
\quad
t > 0, \quad  h \text{ is } 2\pi \text{-periodic in } x. $$
Pukhnachov's 1977 article \cite{Pukh1} gives the first model that
takes into account surface tension:
\begin{equation} \label{A:PukhEq}
h_t + ( h - \tfrac{\mu}{3} h^3 \, \cos(x) )_x + \tfrac{\chi}{3}\,
\left( h^3 \left(h_x + h_{xxx} \right) \right)_x = 0
\end{equation}
where $\mu$ and $\chi$ are given in (\ref{params}) and  $$ x \in
[-\pi,\pi], \quad
t > 0, \quad  h \text{ is } 2\pi \text{-periodic in } x. $$ This
model assumes a no-slip boundary condition at the liquid/solid
interface.  For a solution to (\ref{A:MoffEq}) or (\ref{A:PukhEq})
to be physically relevant, either $h$ is strictly positive (the
cylinder is fully coated) or $h$ is nonnegative (the cylinder is wet
in some region and dry in others).

Surprisingly little is understood about the initial value problem
for (\ref{A:PukhEq}). Bernis and Friedman  \cite{B8} were the first
to prove the existence of nonnegative weak solutions for nonnegative
initial data for the related fourth-order nonlinear degenerate
parabolic PDE
\begin{equation} \label{lube}
h_t + (f(h)\,h_{xxx})_x = 0,
\end{equation}
where $f(h) = |h|^n\,f_0(h), \quad f_0(h) > 0, \quad n \geqslant 1$.

Unlike for second-order parabolic equations, there is no comparison
principle for equation (\ref{lube}).  Nonnegative initial data does
not automatically yield a nonnegative solution; indeed it may not
even be true for general fourth-order PDE (e.g. consider $h_t = -
h_{xxxx}$).  The degeneracy $f(h)$ in equation (\ref{lube}) is key
in ensuring that nonnegative solutions exist.

Lower-order terms can be added to equation (\ref{lube}) to model
additional physical effects.  For example,
\begin{equation} \label{long_wave_stable}
h_t + (f(h)\,h_{xxx})_x  - (g(h) h_x)_x = 0
\end{equation}
where $g(h) > 0$ for $h \neq 0$.  Equation (\ref{long_wave_stable})
can model a thin liquid film on a horizontal surface with gravity
acting towards the surface. If this surface is not horizontal then
the dynamics can be modelled by
\begin{equation} \label{advect}
h_t +(h^n(a - b\,h_x + h_{xxx}))_x = 0, \ \  a > 0, \ \ b \geq 0
\end{equation}
The constant $a$ in the first-order term vanishes as the surface
becomes more and more horizontal. If the thin film of liquid is on a
horizontal surface with gravity acting away from the surface then
the thin film dynamics can be modelled by
\begin{equation} \label{long_wave_unstable}
h_t + (f(h)\,h_{xxx})_x  + (g(h) h_x)_x = 0.
\end{equation}
For a thorough review of the modelling of thin liquid films, see
\cite{Craster,Myers,O1}.

In equations (\ref{long_wave_stable}) and (\ref{advect}) the
second-order term is stabilizing: if one linearizes the equation about
a constant, positive steady state then the presence of the
second-order term increases how quickly perturbations decay in time.
In equation (\ref{long_wave_unstable}), the second-order term is
destabilizing: the linearized equation can have some long-wavelength
perturbations that grow in time.  For this reason, we refer to
equation (\ref{long_wave_unstable}) as ``long--wave unstable''.  The
long--wave stable equations (\ref{long_wave_stable}) and
(\ref{advect}) have similar dynamics as equation (\ref{lube}) however
the long-wave unstable equation (\ref{long_wave_unstable}) can have
nontrivial exact solutions and can have finite--time blow--up
($h(x^*,t) \uparrow \infty$ as $t \uparrow t^* < \infty$).

In all cases, the fourth-order term makes it harder to prove desirable
properties such as: the short--time (or long--time) existence of
nonnegative solutions given nonnegative initial data, compactly
supported initial data yielding compactly supported solutions (finite
speed of propagation), and uniqueness.  Indeed, there are
counterexamples to uniqueness of weak solutions \cite{B2}.  Results
about existence and long--time behavior for solutions of
(\ref{long_wave_stable}) can be found in \cite{B14}; analogous results
for (\ref{advect}) are in \cite{Gi}.  See \cite{B15,BertPughBlow} for
results about existence, finite speed of propagation, and finite--time
blow--up for equation (\ref{long_wave_unstable}).

In this paper we study the existence of  weak solutions of the thin
film equation
\begin{equation} \label{A:mainEq}
h_t  + \left({ |h|^3 (a_0 \, h_{xxx} + a_1 h_x + a_2 w'(x) ) }
\right)_x + a_3 h_x = 0
\end{equation}
where $a_1, \, a_2, \, a_3 $ are arbitrary constants, constant $a_0
> 0$, and $w(x)$ is periodic. Equation (\ref{A:PukhEq}) is a
special case of (\ref{A:mainEq}). The sign of $a_1$ determines
whether equation (\ref{A:mainEq}) is long--wave unstable. Also, the
coefficient of the convection term $a_2 (w'(x) |h|^3)_x$ can depend
on space and will change sign if $a_2 w'(x) \not\equiv 0$. The cubic
nonlinearity $|h|^3$ in equation (\ref{A:mainEq}) arises naturally
in models of thin liquid films with no-slip boundary conditions at
the liquid/solid interface. Our methods generalize naturally to
$f(h)=|h|^n$; we refer the reader to \cite{B2,B8,BertPugh1996} for
the types of results expected.

Given nonnegative initial data that satisfies some reasonable
conditions, we prove long-time existence of nonnegative periodic
generalized weak solutions to the initial value problem for equation
(\ref{A:mainEq}). We start by using energy methods to prove
short-time existence of a weak solution and find an explicit lower
bound on the time of existence. A generalization and sharpening of
the method used in \cite{B15} allows us to prove that the $H^1$ norm
of the constructed solution can grow at most linearly in time,
precluding the possibility of a finite--time blow--up. This $H^1$
control, combined with the explicit lower bound on the (short) time
of existence, allows us to continue the weak solution in time,
extending the short-time result to a long-time result.

If $a_2=0$ or $a_3 = 0$ in equation (\ref{A:mainEq}) then solutions
will be uniformly bounded for all time.  If $a_2 \neq 0$ and $a_3
\neq 0$, it is natural to ask if the nonlinear advection term could
cause finite--time blow--up ($h(x^*,t) \uparrow \infty$ as $t
\uparrow t^*$. Such finite-time blow-up is impossible by the
linear-in-time bound on $H^1$ but we have not ruled out that a
solution might grow in an unbounded manner as time goes to infinity.

In \cite{BDGG, DGG}, the authors consider
the multidimensional analogue of (\ref{lube})
\begin{equation} \label{M1}
h_t  + \nabla \cdot \left( {|h|^n \nabla \Delta h} \right) = 0,
\end{equation}
for $h(x,t)$ where $x \in \Omega \subset \R^N$ with $N = 2,3$.
Depending on the sign of $A'$, if $g=0$ then equation
\begin{equation} \label{M2}
h_t  +
\nabla \cdot \left( {f(h)\nabla \Delta h + \nabla A(h)}
\right) = g(t,x,h,\nabla h)
\end{equation}
on $\Omega$
is the multidimensional analogue of equation (\ref{long_wave_stable})
or (\ref{long_wave_unstable}).
In \cite{D3}, the authors consider the long-wave
stable case with $g = 0$ and power-law coefficients, $f(h) = |h|^n$ and
$A'(h) = - |h|^m$.
In \cite{Gr}, the author considers the Neumann problem for
both the long-wave stable and unstable cases
with the assumption that $f(h) \geq 0$ has power-law-like behavior
near $h=0$, that
$|A'(h)|$ is dominated by $f(h)$ (specifically
$|A'(h)| \leqslant d_0 f(h) $ for some $d_0$), and that the source/sink term $g(t,x,h)$ grows no
faster than linearly in $h$.
In  \cite{T1,T2,T3}, the authors consider the Neumann problem for
the long-wave stable case of (\ref{M2}) with power-law coefficients
and a larger class of source
terms: $g(t,x,h) \sim |h|^{\lambda -1}h$ with $\lambda > 0$.
In \cite{T4,T5}, the same authors consider the long-wave stable equation with power-law
coefficients but with
$g(h) =  \vec{a} \cdot \nabla b(h)$ where $b(z) \sim z^{\lambda}$
and $\vec{a} \in \mathbb{R}^N$: $g$ models advective effects.  They
consider the problem both  on $\R^N$ and on a bounded domain $\Omega$.

All of these works on (\ref{M1}) and (\ref{M2}) construct nonnegative weak solutions from nonnegative
initial data and address qualitative questions such as dependence on exponents $n$ and
$m$ and $\lambda$, on dimension $N$, speed of propagation of the support and of
perturbations, exact asymptotics of the motion of the support, and positivity properties.
We note that the works \cite{T1,T2,T3,T4,T5} also construct ``strong'' solutions.

Finally, we refer readers to the technical report \cite{Report}
which presents the results of this article, and some additional
results, along with more extensive discussion, calculations, and
simulations.

\section{Steady state solutions} \label{AA}

Smooth steady state solutions, $h(x,t) = h(x)$, of (\ref{A:PukhEq}) satisfy
\begin{equation} \label{chi_nonzero_ss}
h - \tfrac{\mu}{3} h^3 \, \cos(x)  + \tfrac{\chi}{3}\,
\left( h^3 \left(h_x + h_{xxx} \right) \right) = q
\end{equation}
where $q$ is a constant of integration that corresponds to the dimensionless mass flux.
In the zero surface tension case ($\chi = 0$), steady states satisfy
\begin{equation} \label{chi_zero_ss}
h - \tfrac{\mu}{3} h^3 \, \cos(x)  = q.
\end{equation}
Such steady states were first studied by Johnson \cite{John} and
Moffatt \cite{Moff}. Johnson proved that there are positive, unique,
smooth steady states if and only if the flux is not too large: $0 <
q < 2/(3 \sqrt{\mu})$. These steady-states are neutrally stable
\cite{OBrien}.
Smooth, positive steady states in the presence of surface tension
have been studied by a number of authors.  One striking
computational result \cite{Benilov3} is that for certain values of
$\chi$ and $\mu$ there can be non-uniqueness.

These non-unique steady states were numerically discovered via an
elegant combination of asymptotics and a two-parameter (mass and
flux) continuation method \cite[Figure 14]{Benilov3}. To start the
continuation method, earlier work \cite{Benilov1} on the regime in
which viscous forces dominate gravity was used.  There, asymptotics
show that for small fluxes the steady state is close to $q + 1/3 q^3
\cos(x) + \mathcal{O}(q^5)$, providing a good first guess for the
iteration used to find the steady state. The bifurcation diagram
shown in Figure 14 of \cite{Benilov3} also suggests that the Moffatt
model (\ref{A:MoffEq}) can be considered as the limit of the
Pukhnachov model (\ref{A:PukhEq}) as surface tension goes to zero
($\chi \to 0$).

Pukhnachov proved \cite{Pukh3} a nonexistence result: no positive
steady states exist if $q > 2 \sqrt{3/\mu} \simeq 3.464/\sqrt{\mu}$.
We improve this, proving that no such solution exists if $q > 2/3 \,
\sqrt{2/\mu} \simeq 0.943/\sqrt{\mu}$.
\begin{proposition}\label{Prop2.1}
There does not exist a strictly positive $2 \pi$ periodic solution
$h(x)$ of equation (\ref{chi_nonzero_ss}) if $q > 2/3 \,
\sqrt{2/\mu}$.
\end{proposition}

\begin{proof}[Proof of Proposition~\ref{Prop2.1}]
Following Pukhnachov, we start by rescaling the flux to $1$ by
introducing $y(x) = h(x)/q$ and introducing the parameters $\gamma =
\tfrac{\chi \, q^3}{3}$ and $\beta = \tfrac{q^2 \mu}{3}$. Equation
(\ref{chi_nonzero_ss}) transforms to
\begin{equation}\label{AA:1}
\gamma (y''' + y') = \beta \cos{(x)} - \tfrac{1}{y^2} +
\tfrac{1}{y^3}.
\end{equation}
The solution $y$ is written as $y(x) = a_0 + a_1 \cos(x) + a_2
\sin(x) + v(x)$ where $v(x) \perp \mbox{span}\{ 1, \cos(x), \sin(x)
\}$ and satisfies
\begin{equation} \label{AA:2}
\gamma (v''' + v') = \beta \cos{(x)} - \tfrac{1}{y(x)^2} +
\tfrac{1}{y(x)^3}.
\end{equation}
A solution $v$ exists only if the right-hand side of (\ref{AA:2}) is
orthogonal to $\mbox{span}\{ 1, \cos(x), \sin(x) \}$.  As a result,
\begin{equation} \label{AA:2=3}
\int\limits_{-\pi}^{\pi} \left( \tfrac{1}{y(x)^2} -
\tfrac{1}{y(x)^3} \right) \; dx = 0,  \  \int\limits_{-\pi}^{\pi}
\left( \tfrac{1}{y(x)^2} - \tfrac{1}{y(x)^3} \right) \cos(x) \; dx =
\pi \, \beta.
\end{equation}
$$\mbox{It follows from (\ref{AA:2=3}) that }\quad \pi \beta \leq
\int\limits_{y \geq 1} \tfrac{4}{27} \left( 1 + \cos(x) \right) \;
dx \leq \tfrac{4}{27} \, 2 \pi.  \quad $$ This shows that if there
is a positive steady state then $\beta \leq 8/27$. Recalling the
definition of $\beta$, there is no steady state if $q
> 2/3 \, \sqrt{2/\mu}$.
\end{proof}

The proof also holds in the case of zero surface tension $\chi =
\gamma = 0$ and so it is natural that the bound $2/3 \, \sqrt{2/\mu}$
is larger than $2/(3 \sqrt{\mu})$ (the bound found by Johnson and
Moffatt.)  Also, we note that numerical simulations that suggest
nonexistence of a positive steady state if $q > 0.854$ when $\mu = 1$
for a large range of surface tension values \cite[p. 61]{Kar}; our bound of
$0.943$ is not too far off from this.

\section{Short--time Existence and Regularity of Solutions} \label{B}

We are interested in the existence of nonnegative generalized weak
solutions to the following initial--boundary value problem:
\begin{numcases}
{(\textup{P})}
h_t  + \left({ f(h) (a_0h_{xxx} + a_1 h_x + a_2 w'(x) ) } \right)_x + a_3 h_x = 0 \text{ in }Q_T, \qquad \quad \label{B:1} \\
\tfrac{\partial^{i} h}{\partial x^i}(-a,t) = \tfrac{\partial^{i}h}{\partial x^i}(a,t) \text{ for }t > 0,\,i=\overline{0,3}, \label{B:2}  \\
h(x,0) = h_0 (x) \geqslant 0, \label{B:3}
\end{numcases}
where $f(h) = |h|^3$, $h= h(x,t)$, $\Omega = (-a, a)$, and $Q_T =
\Omega \times (0,T)$. Note that rather than considering the interval
$(-a,a)$ with boundary conditions (\ref{B:2}) one can equally well
consider the problem on the circle $S^1$; our methods and results
would apply here too. Recall that $a_1$, $a_2$, and $a_3$ in
equation (\ref{B:1}) are arbitrary constants; $a_0$ is required to
be positive.  The function $w$ in (\ref{B:1}) is assumed to satisfy:
\begin{equation}\label{B:w}
w \in C^{2+\gamma}(\Omega) \; \mbox{for some} \; 0 < \gamma < 1,
\tfrac{\partial^{i} w}{\partial x^i}(-a) =
\tfrac{\partial^{i}w}{\partial x^i}(a) \text{ for }i =
\overline{0,2}.
\end{equation}
We consider a generalized weak solution in the following sense \cite{B2,B6}:

\begin{definition}\label{B:defweak}
A generalized weak solution of problem $(\textup{P})$ is a function
$h$ satisfying
\begin{align}
& \label{weak1}
h \in C^{1/2,1/8}_{x,t}(\overline{Q}_T) \cap L^\infty (0,T; H^1(\Omega )),\\
& \label{weak-d} h_t \in L^2(0,T; (H^1(\Omega))'),\\
& \label{weak2} h \in C^{4,1}_{x,t}(\mathcal{P}), \,\,\, \sqrt{f(h)}
\, \left( a_0 h_{xxx} + a_1 h_x + a_2 w' \right) \in
L^2(\mathcal{P}), \,\,
\end{align}
where $\mathcal{P} = \overline{Q}_T \setminus ( {h=0} \cup {t=0})$
and $h$ satisfies (\ref{B:1}) in the following sense:
\begin{align}\notag
& \int\limits_0^T \langle h_t(\cdot,t), \phi \rangle \; dt -
\iint\limits_{\mathcal{P}}
{f(h) ( a_0h_{xxx} + a_1 h_x + a_2 w'(x))\phi_x\,dx dt } \\
& \hspace{2.5in} - a_3\iint\limits_{Q_{T}} {h \phi_x\,dx dt} = 0
\label{integral_form}
\end{align}
for all $\phi \in C^1(Q_T)$ with $\phi(-a,\cdot) = \phi(a,\cdot)$;
\begin{align}
&  \label{ID1} h(\cdot,t) \to h(\cdot,0) = h_0
\mbox{  pointwise \& strongly in $L^2(\Omega)$ as $t \to 0$}, \\
& \label{BC1} h(-a,t)=h(a,t) \; \forall t \in [0,T] \; \mbox{and} \;
\tfrac{\partial^{i}
h}{\partial x^i}(-a,t) = \tfrac{\partial^{i}h}{\partial x^i}(a,t)  \\
& \notag \mbox{for} \; i = \overline{1,3} \; \mbox{at all points of
the lateral boundary where $\{h \neq 0\}$.}
\end{align}
\end{definition}

Because the second term of (\ref{integral_form}) has an integral
over $\mathcal{P}$ rather than over $Q_T$, the generalized weak
solution is ``weaker'' than a standard weak solution. Also note that
the first term of (\ref{integral_form}) uses $h_t \in L^2(0,T;
(H^1(\Omega))' )$; this is different from the definition of weak
solution first introduced by Bernis and Friedman \cite{B8}; there,
the first term was the integral of $h \phi_t$ integrated over $Q_T$.

We first prove the short-time existence of a generalized weak
solution and then prove that it can have additional regularity.  In
Section \ref{G} we prove additional control for the $H^1$ norm which
then allows us to prove long-time existence.

\begin{theorem}[Existence]\label{C:Th1}
Let the nonnegative initial data $h_0 \in H^1(\Omega)$ satisfy
\begin{equation}\label{C:inval}
\int\limits_{\Omega} {\tfrac{1}{h_0(x)}} \; dx < \infty,
\end{equation}
and either 1) $h_0(-a) = h_0(a) = 0$ or  2) $h_0(-a) = h_0(a) \neq
0$ and $ \tfrac{\partial^{i} h_0}{\partial x^i}(-a) =
\tfrac{\partial^{i}h_0}{\partial x^i}(a) \text{ holds for } i =
\overline{1,3}$. Then for some time $T_{loc}>0$ there exists a
nonnegative generalized weak solution, $h$, on $Q_{T_{loc}}$ in the
sense of the definition \ref{B:defweak}.  Furthermore,
\begin{equation} \label{Linf_H2}
h \in L^2(0,T_{loc};H^2(\Omega)).
\end{equation}
Let
\begin{equation} \label{Energy}
\mathcal{E}_0(T) := \tfrac{1}{2}\int\limits_{\Omega} ({  a_0
h_x^2(x,T) - a_1 h^2(x,T) - 2a_2 w(x) h(x,T)) \,dx}
\end{equation}
then the weak solution satisfies
\begin{equation}\label{C:d2'}
\mathcal{E}_0(T_{loc}) +
\iint\limits_{\{h >0 \}}
{h^3 (a_0 h_{xxx} +
a_1 h_x  + a_2 w')^2 \,dx \, dt} \leqslant \mathcal{E}_0(0) + K\, T_{loc},
\end{equation}
where $K = |a_2a_3| \, \| w' \|_\infty C < \infty$. The time of
existence, $T_{loc}$, is determined by $a_0$, $a_1$, $a_2$, $w'$,
$|\Omega |$, and $h_0$.
\end{theorem}

We note that the analogue of Theorem 4.2 in \cite{B8} also holds:
there exists a nonnegative weak solution with the integral
formulation
\begin{align} \label{alt_int}
& \int\limits_0^T \langle h_t(\cdot,t), \phi \rangle \; dt
+ a_0 \iint\limits_{Q_T} (3 h^2 h_x h_{xx} \phi_x + h^3 h_{xx} \phi_{xx}) \; dx dt \\
& \hspace{1.5in} - \iint\limits_{Q_T} \left( a_1 h_x + a_2 w' + a_3
h \right) \phi_x \; dx dt = 0. \notag
\end{align}

\begin{theorem}[Regularity]\label{C:Th1.3}
If the initial data from Theorem \ref{C:Th1} also satisfies
$$
\int\limits_{\Omega} {h_0^{\alpha - 1}(x) \,dx} <\infty
$$
for some $-1/2<\alpha<1, \,\,\alpha\neq0$ then there exists $0<
T_{loc}^{(\alpha)}\leq T_{loc}$ such that the nonnegative
generalized weak solution from Theorem \ref{C:Th1} has the extra
regularity $$h^{\tfrac{\alpha + 2}{2}} \in L^{2}(0,
T_{loc}^{(\alpha)}; H^2(\Omega))\quad \mbox{and} \quad
h^{\tfrac{\alpha + 2}{4}} \in L^{2}(0, T_{loc}^{(\alpha)};
W^1_4(\Omega)).$$
\end{theorem}

The solutions from Theorem \ref{C:Th1.3} are often called ``strong''
solutions in the thin film literature. If the initial data satisfies
$\int h_0^{\alpha-1} \; dx < \infty$ then the added regularity from
Theorem \ref{C:Th1.3} allows one to prove the existence of
nonnegative solutions with an integral formulation
\cite{BertPugh1996} that is similar to that of (\ref{alt_int})
except that the second integral is replaced by the results of one
more integration by parts (there are no $h_{xx}$ terms). We also
note that if one considered problem (P) with nonlinearity $f(h) =
|h|^n$ with $0 < n < 3$, then Theorems \ref{C:Th1} and \ref{C:Th1.3}
would hold for general nonnegative initial data $h_0 \in
H^1(\Omega)$; no ``finite entropy'' assumption would be needed
\cite{BertPugh1996, B2}. Finite entropy conditions ($\int
h_0^{2-n}\; dx < \infty$ and $\int h_0^{\alpha+2-n}\; dx < \infty$)
would be needed to obtain the results for $n \geq 3$.

\subsection{Regularized Problem}

\label{RegularizedProblem}

Given $\delta, \eps > 0$, a regularized parabolic
problem, similar to that of Bernis and
Friedman \cite{B8},
is considered:
\begin{numcases}
{(\textup{P}_{\delta,\epsilon})}
 h_t  + \left( { f_{\delta \eps}(h)
\bigl(a_0 h_{xxx} + a_1 h_x + a_2 w'(x) \bigr) }
\right)_x + a_3 h_x = 0, \qquad \quad \hfill \label{D:1r'}\\
\tfrac{\partial^{i} h}{\partial x^i}(-a,t) = \tfrac{\partial^{i}
h}{\partial x^i}(a,t) \text{ for }
t > 0,\, i = \overline{0,3} , \hfill \label{D:2r'}\\
\qquad  \qquad h(x,0) = h_{0,\delta \eps}(x), \hfill \label{D:3r'}
\end{numcases}
where
\begin{equation}\label{D:reg1}
f_{\delta \varepsilon} (z) := f_\eps(z) + \delta =
\tfrac{|z|^4}{|z| + \varepsilon} + \delta \quad \ \forall\, z \in
\mathbb{R}^1,\ \delta>0,\ \varepsilon >0.
\end{equation}
The $\delta>0$ in (\ref{D:reg1}) makes the problem (\ref{D:1r'})
regular (i.e. uniformly parabolic). The parameter $\eps$ is an
approximating parameter which has the effect of increasing the
degeneracy from $f(h) \sim |h|^3$ to $f_{\eps}(h) \sim h^4$. The
nonnegative initial data, $h_0$, is approximated via
\begin{equation}\label{D:inreg}
\begin{gathered}
h_{0,\delta \eps} =
h_{0,\delta} + \eps^\theta \in C^{4+\gamma}(\Omega)
\text{ for some } \theta \in (0,2/5) \text{ and } \gamma  \text{ from }  (\ref{B:w} )\\
\tfrac{\partial^{i} h_{0,\delta \eps}}{\partial x^i}(-a) =
\tfrac{\partial^{i}h_{0,\delta \eps}}{\partial x^i}(a)
\text{ for } i=\overline{0,3}, \\
h_{0,\delta \varepsilon} \to h_{0}  \text{ strongly in } H^1(\Omega)
\text{ as } \delta, \varepsilon \to 0.
\end{gathered}
\end{equation}
 The $\varepsilon$ term in (\ref{D:inreg}) ``lifts'' the initial data so that it will be positive even if
$\delta = 0$ and the $\delta$ is involved in
smoothing the initial data from $H^1(\Omega)$ to $C^{4+\gamma}(\Omega)$.

By E\u{i}delman \cite[Theorem 6.3, p.302]{Ed}, the regularized
problem has a unique classical solution $h_{\delta \eps} \in
C_{x,t}^{4+\gamma,1+\gamma/4}( \Omega \times [0, \tau_{\delta
\eps}])$ for some time $\tau_{\delta \eps} > 0$.
For any fixed value of  $\delta$ and $\eps$, by E\u{i}delman \cite
[Theorem 9.3, p.316]{Ed} if one can prove an uniform in time a
priori bound $|h_{\delta \eps}(x,t)| \leq A_{\delta \eps}<\infty$
for some longer time interval $[0,T_{loc,\delta \eps}] \quad
(T_{loc,\delta \eps} > \tau_{\delta \eps}$) and for all $x \in
\Omega$ then Schauder-type interior estimates \cite [Corollary 2,
p.213] {Ed} imply that the solution $h_{\delta \eps}$ can be
continued in time to be in $C_{x,t}^{4+\gamma,1+\gamma/4}( \Omega
\times [0,T_{loc,\delta \eps}])$.

Although the solution $h_{\delta \eps}$ is initially positive, there
is no guarantee that it will remain nonnegative. The goal is to take
$\delta \to 0$, $\epsilon \to 0$ in such a way that 1)
$T_{loc,\delta \eps} \to T_{loc} > 0$, 2) the solutions $h_{\delta
\eps}$ converge to a (nonnegative) limit, $h$, which is a
generalized weak solution, and 3) $h$ inherits certain a priori
bounds.  This is done by proving various a priori estimates for
$h_{\delta \eps}$ that are uniform in  $\delta$ and $\eps$ and hold
on a time interval $[0,T_{loc}]$ that is independent of $\delta$ and
$\eps$.  As a result, $\{ h_{\delta \eps} \}$ will be a uniformly
bounded and equicontinuous (in the $C_{x,t}^{1/2,1/8}$ norm) family
of functions in $\bar{\Omega} \times [0, T_{loc}]$. Taking $\delta
\to 0$ will result in a  family of functions $\{ h_{\eps} \}$ that
are classical, positive, unique solutions to the regularized problem
with $\delta = 0$.   Taking $\eps \to 0$ will then result in the
desired generalized weak solution $h$.  This last  step is where the
possibility of nonunique weak solutions arise; see \cite{B2} for
simple examples of how such constructions applied to $h_t = - (|h|^n
h_{xxx})_x$ can result in two different solutions arising from the
same initial data.

\subsection{A priori estimates}

Our first task is to derive a priori estimates for classical
solutions of (\ref{D:1r'})--(\ref{D:inreg}). The lemmas in this
section are proved in Section \ref{A_priori_proofs}.

We use an integral quantity based on a function
$G_{\delta \eps}$ chosen so that
\begin{equation}\label{D:reg2}
G''_{\delta \varepsilon} (z) = \tfrac{1}{f_{\delta \varepsilon} (z)}
\quad \mbox{and} \quad G_{\delta \eps}(z) \geq 0.
\end{equation}
This is analogous to the ``entropy'' function first introduced by
Bernis and Friedman \cite{B8}.

\begin{lemma}\label{MainAE}
There exists $\delta_0 > 0$, $\eps_0 > 0$, and time $T_{loc}>0$ such
that if $\delta \in [0,\delta_0)$, $\eps \in (0,\eps_0)$, if
$h_{\delta \eps}$ is a classical solution of the problem
(\ref{D:1r'})--(\ref{D:inreg}) with initial data $h_{0,\delta
\eps}$, and if $h_{0,\delta \eps}$ satisfies (\ref{D:inreg}) and is
built from a nonnegative function $h_0$ that satisfies the
hypotheses of Theorem \ref{C:Th1} then for any $T \in [0, T_{loc}]$
the solution $h_{\delta \eps}$ satisfies
\begin{align}\label{D:a13''}
& \int\limits_{\Omega} { \{h_{\delta \eps, x}^2(x,T) + \tfrac{|a_1|}{a_0}\left( \tfrac{|a_1|}{a_0}
+ 2 \delta  \right) G_{\delta \varepsilon}(h_{\delta \eps}(x,T))\} \,dx} \\
& \hspace{1in}+ \notag a_0 \iint\limits_{Q_T} { f_{\delta
\varepsilon}(h_{\delta \eps}) h^2_{\delta \eps, xxx}  \,dx dt}
\leqslant K_1 < \infty,
\end{align}
\begin{equation} \label{BF_entropy}
\int\limits_\Omega G_{\delta \eps}(h_{\delta \eps}(x,T)) \; dx + a_0 \iint\limits_{Q_T} h_{\delta \eps, xx}^2 \; dx dt
\leq K_2 < \infty,
\end{equation}
and the energy
$\mathcal{E}_{\delta \varepsilon} (t)$ (see (\ref{Energy})) satisfies:
\begin{align}\label{D:d2}
& \mathcal{E}_{\delta \varepsilon}(T) + \iint\limits_{Q_T}
{f_{\delta \varepsilon}(h_{\delta \eps}) (a_0 h_{\delta \eps,xxx} + a_1 h_{\delta \eps, x} + a_2 w')^2} \; dx dt \\
& \hspace{3in} \notag
\leqslant  C_0 + K_3 T
\end{align}
where $K_3 =|a_2 a_3|\mathop \| w' \|_\infty C < \infty$. The time
$T_{loc}$ and the constants $K_1$, $K_2$, $C_0$, and $K_3$ are
independent of $\delta$ and $\eps$.

\end{lemma}
The existence of $\delta_0$, $\eps_0$, $T_{loc}$, $K_1$, $K_2$, and
$K_3$ is constructive; how to find them and what quantities
determine them is shown in Section \ref{A_priori_proofs}.

Lemma \ref{MainAE} yields uniform-in-$\delta$-and-$\eps$ bounds for
$\int h_{\delta \eps,x}^2$, $\int G_{\delta \eps}(h_{\delta \eps})$,
$\iint h_{\delta \eps,xx}^2$, and $\iint f_{\delta \eps}(h_{\delta
\eps}) h_{\delta \eps,xxx}^2$.  However, these bounds are found in a
different manner than in earlier work for the equation $h_t = -
(|h|^n h_{xxx})_x$, for example.  Although the inequality
(\ref{BF_entropy}) is unchanged, the inequality (\ref{D:a13''}) has
an extra term involving $G_{\delta \eps}$.  In the proof, this term
was introduced to control additional, lower--order terms. This idea
of a ``blended'' $\| h_x \|_2$--entropy bound was first introduced
by Shishkov and Taranets especially for long-wave stable thin film
equations with convection \cite{T4}.

The final a priori bound uses the following functions, parametrized
by $\alpha$,
\begin{equation}\label{E:reg4}
G_{\varepsilon}^{(\alpha)} (z): =\tfrac{z^{\alpha -
1}}{(\alpha-1)(\alpha - 2)} + \tfrac{\varepsilon z^{\alpha
-2}}{(\alpha - 3)(\alpha - 2)}; \ (G^{(\alpha)}_{\varepsilon} (z))''
= \tfrac{z^{\alpha}}{f_{\varepsilon} (z)}.
\end{equation}

\begin{lemma}\label{MainAE2}
Assume $\eps_0$ and $T_{loc}$ are from Lemma \ref{MainAE}, $\delta =
0$, and $\eps \in (0,\eps_0)$.  Assume $h_\eps$ is a positive,
classical solution of the problem (\ref{D:1r'})--(\ref{D:inreg})
with initial data $h_{0,\eps}$ satisfying Lemma \ref{MainAE}.
Fix
$\alpha \in (-1/2,1)$ with $\alpha \neq 0$. If the initial data
$h_{0,\eps}$ is built from $h_0$ which also satisfies
\begin{equation} \label{finite_alpha_ent}
\int\limits_\Omega h_0^{\alpha - 1}(x) \; dx < \infty
\end{equation}
then there exists $\eps_0^{(\alpha)}$ and $T_{loc}^{(\alpha)}$ with
$0 < \eps_0^{(\alpha)} \leq \eps_0$ and $0 < T_{loc}^{(\alpha)} \leq
T_{loc}$ such that
\begin{align}
\label{E:b12''}
& \int\limits_{\Omega} { \{h_{\eps,x}^2(x,T) + G_{\varepsilon}^{(\alpha)}(h_\eps(x,T))\}
\,dx} \\
& \hspace{1in} \notag + \iint\limits_{Q_T} \left[ \beta
h_\eps^\alpha h_{\eps,xx}^2 + \gamma h_\eps^{\alpha-2} h_{\eps,x}^4
\right]\;dx\,dt \leqslant K_4 < \infty
\end{align}
holds for all $T \in [0 , T_{loc}^{(\alpha)} ]$ and some constant
$K_4$ that is determined by $\alpha$, $\eps_0$, $a_0$, $a_1$, $a_2$,
$w'$, $\Omega$ and $h_0$. Here,
$$
\beta =
\begin{cases}
a_0 & \mbox{if } \alpha \in (0,1), \\
a_0 \tfrac{1+2\alpha}{4(1-\alpha)} &\mbox{if } \alpha \in (-1/2,0) ,
\end{cases}
\,\, \gamma =
\begin{cases}
a_0 \tfrac{\alpha(1-\alpha)}{6}&\mbox{if } \alpha \in (0,1) , \\
a_0 \tfrac{(1+2\alpha)(1-\alpha)}{36} & \mbox{if } \alpha \in
(-1/2,0).
\end{cases}
$$
 Furthermore,
\begin{equation}\label{E:b13}
h_\eps^{\tfrac{\alpha+2}{2}} \in L^{2}(0, T_{loc}; H^2(\Omega))
\quad \mbox{and} \quad
h_\eps^{\tfrac{\alpha+2}{4}} \in L^{2}(0,
T_{loc}; W^1_4(\Omega))
\end{equation}
with a uniform-in-$\eps$ bound.
\end{lemma}

The $\alpha$--entropy, $\int G_0^{(\alpha)}(h) \; dx$, was first
introduced for $\alpha = - 1/2$ in \cite{BertozziBrenner} and an a
priori bound like that of Lemma \ref{MainAE2} and regularity results
like those of Theorem \ref{C:Th1.3} were found simultaneously and
independently in \cite{B2} and \cite{BertPugh1996}.

\subsection{Proof of existence and regularity of solutions}
Bound (\ref{D:a13''}) yields uniform  $L^\infty$ control for
classical solutions $h_{\delta \eps}$, allowing the time of
existence $T_{loc,\delta \eps}$ to be taken as $T_{loc}$ for all
$\delta \in (0,\delta_0)$ and $\eps \in (0,\eps_0)$. The existence
theory starts by constructing a classical solution $h_{\delta \eps}$
on $[0,T_{loc}]$ that satisfy the hypotheses of Lemma \ref{MainAE}
if $\delta \in (0,\delta_0)$ and $\eps \in (0,\eps_0)$. The
regularizing parameter, $\delta$, is taken to zero and one proves
that there is a limit $h_\eps$ and that $h_\eps$ is a generalized
weak solution. One then proves additional regularity for $h_\eps$;
specifically that it is strictly positive, classical, and unique. It
then follows that the a priori bounds given by Lemmas \ref{MainAE},
and \ref{MainAE2} apply to $h_\eps$. This allows us to take the
approximating parameter, $\eps$, to zero and construct the desired
generalized weak solution of Theorems \ref{C:Th1} and \ref{C:Th1.3}.

\begin{lemma}
\label{preC:Th1} Assume that the initial data $h_{0,\eps}$ satisfies
(\ref{D:inreg}) and is built from a nonnegative function $h_0$ that
satisfies the hypotheses of Theorem \ref{C:Th1}. Fix  $\delta = 0$
and $\eps \in (0,\eps_0)$ where $\eps_0$ is from Lemma \ref{MainAE}.
Then there exists a unique, positive, classical solution $h_\eps$ on
$[0,T_{loc}]$ of problem ($\mbox{P}_{0, \eps}$), see
(\ref{D:1r'})--(\ref{D:inreg}), with initial data $h_{0,\eps}$ where
$T_{loc}$ is the time from Lemma \ref{MainAE}.
\end{lemma}

\begin{proof}
Arguing the same way as Bernis \& Friedman \cite{B8} one can
construct a generalized weak solution $h_\eps$. We now prove that
this $h_\eps$ is a strictly positive, classical, unique solution.
This uses the entropy $\int G_{\delta \eps}(h_{\delta \eps})$ and
the a priori bound (\ref{BF_entropy}). This bound  is, up to the
coefficient $a_0$, identical to the a priori bound (4.17) in
\cite{B8}. By construction, the initial data $h_{0,\eps}$ is
positive (see (\ref{D:inreg})), hence $\int G_{\eps}(h_{0,\eps}) \;
dx < \infty$. Also, by construction $f_\eps(z) \sim z^4$ for $z \ll
1$.  This implies that the generalized weak solution $h_\eps$ is
strictly positive \cite[Theorem 4.1]{B8}.  Because the initial data
$h_{0,\eps}$ is in $C^4(\Omega)$, it follows that $h_\eps$ is a
classical solution in $C^{4,1}_{x,t}(\overline{Q_{T_{loc}}})$. The
proof of  Theorem 4.1 in \cite{B8} then implies that $h_\eps$ is
unique.
\end{proof}

\begin{proof}[Proof of Theorem \ref{C:Th1}]
As in the proof of Lemma \ref{preC:Th1}, following \cite{B8}, there
is a subsequence $\{ \eps_k \}$ such that $h_{\eps_k}$ converges
uniformly to a function $h \in C^{1/2,1/8}_{x,t}$ which is a
generalized weak solution in the sense of Definition \ref{B:defweak}
with $f(h) = |h|^3$.

The initial data is assumed to have finite entropy: $\int 1/h_0 < \infty$.
This, combined with $f(h) = |h|^3$, implies that the
generalized weak solution $h$ is nonnegative and the set of points $\{ h = 0 \}$
in $Q_{T_{loc}}$
has zero measure  \cite[Theorem 4.1]{B8}.

To prove (\ref{C:d2'}), start by taking $T=T_{loc}$ in the a priori
bound (\ref{D:d2}). As $\eps_k \to 0$, the right-hand side of
(\ref{D:d2}) is unchanged. First, consider the $\eps_k \to 0$ limit
of
$$
\mathcal{E}_{\eps_k}(T_{loc})
=
 \tfrac{1}{2}\int\limits_{\Omega} {a_0
h_{\eps_k,x}^2(x,T_{loc}) - a_1 h_{\eps_k}^2(x,T_{loc}) - 2a_2 w(x)
h_{\eps_k}(x,T_{loc})dx}.
$$
By the uniform convergence of $h_{\eps_k}$ to $h$, the second and
third terms in the energy converge strongly as $\eps_k \to 0$. The
bound (\ref{D:d2}) yields a uniform bound on $\{   \int_\Omega
h_{\eps_k,x}^2(x,T_{loc}) \; dx \}$. Taking a further refinement of
$\{ \eps_k \}$, yields $h_{\eps_k,x}(\cdot,T_{loc})$ converging
weakly in $L^2(\Omega)$.  In a Hilbert space, the norm of the weak
limit is less than or equal to the $\liminf$ of the norms of the
functions in the sequence, hence $ \mathcal{E}_0( T_{loc} ) \leq
\liminf_{\eps_k \to 0} \mathcal{E}_{\eps_k}(T_{loc}). $ A uniform
bound on $\iint f_\eps(h_\eps) \left(a_0 h_{\eps,xxx} + \dots
\right)^2 \; dx$ also follows from (\ref{D:d2}).  Hence
$\sqrt{f_{\eps_k}(h_{\eps_k})} \left(a_0 h_{\eps_k,xxx} + \dots
\right)$ converges weakly in $L^2(Q_{T_{loc}})$, after taking a
further subsequence. It suffices to determine the weak limit up to a
set of measure zero. Because $h \geq 0$ and $\{ h = 0 \}$ has
measure zero, it suffices to determine the weak limit on $\{ h > 0
\}$.

The regularity theory for parabolic equations allows one to argue
that $h \in C^{4,1}_{x,t}(\mathcal{P})$, and the weak limit is $
h^{3/2} \left(a_0 h_{xxx} + \dots \right)$ on $\{ h > 0 \}$. Using
that 1) the norm of the weak limit is less than or equal to the
$\liminf$ of the norms of the functions in the sequence and that 2)
the $\liminf$ of a sum is greater than or equal to the sum of the
$\liminf$s, results in the desired bound (\ref{C:d2'}).

It follows from (\ref{BF_entropy}) that $h_{\eps_k,xx}$ converges
weakly to some $v$ in $L^2(Q_{T_{loc}})$, combining with strong
convergence in $L^2(0,T; H^1(\Omega ))$ of $h_{\eps_k}$ to $h$ by
Lemma \ref{A.1} and with the definition of weak derivative, we
obtain that $v = h_{xx}$ and $h \in L^2(0,T_{loc};$ $ H^2(\Omega))$
that implies (\ref{Linf_H2}). Hence $h_{\varepsilon,t} \to h_t
\text{ weakly in } L^{2}(0, T; (H^1(\Omega))')$ that implies
(\ref{weak-d}). By Lemma \ref{A.2} we also have $h \in
C([0,T_{loc}],L^2(\Omega))$.
\end{proof}

\begin{proof}[Proof of Theorem \ref{C:Th1.3}]
Fix $\alpha \in (-1/2,1)$. The initial data $h_0$ is assumed to have
finite entropy $\int G_0^{(\alpha)}(h_0(x)) \; dx < \infty$, hence
Lemma \ref{MainAE2} holds for the approximate solutions $\{
h_{\eps_k} \}$ where this sequence of approximate solutions is
assumed to be the one at the end of the proof of
Theorem~\ref{C:Th1}. By (\ref{E:b13}),
$$
\left\{ h_{\eps_k}^{\tfrac{\alpha+2}{2}} \right\}
\quad \mbox{is uniformly bounded in $\eps_k$ in
$L^{2}(0, T_{loc}; H^2(\Omega))$}
$$
and
$$
\left\{
h_{\eps_k}^{\tfrac{\alpha+2}{4}} \right\}
\quad \mbox{is uniformly bounded in $\eps_k$ in
$L^{2}(0, T_{loc}; W^1_4(\Omega))$}.
$$
Taking a further subsequence in $\{ \eps_k \}$, it follows from the
proof of \cite[Lemma 2.5, p.330]{DGG}, these sequences converge
weakly in $L^{2}(0, T_{loc}; H^2(\Omega))$ and $L^{2}(0,
T_{loc};W^1_4(\Omega))$, to $h^{\tfrac{\alpha+2}{2}}$ and $
h^{\tfrac{\alpha+2}{4}}$ respectively.
\end{proof}

\section{Long--time existence of solutions} \label{G}

\begin{lemma}\label{Marinas_bound}
Let $h \in H^1(\Omega)$ be a nonnegative function such that $\int
\limits_{\Omega} {h(x)\,dx} = M > 0.$ Then
\begin{equation}\label{D:nint}
 \|h \|_{L^2(\Omega)}^2 \leqslant 6^{\tfrac{2}{3}}M^{\tfrac{4}{3}}
\biggl(\int\limits_{\Omega} {h^2_{x} \,dx}\biggr)^{\tfrac{1}{3}} +
\tfrac{M ^{2}}{|\Omega|}.
\end{equation}
\end{lemma}
Note that by taking $h$ to be a constant function, one finds
that the constant $M^2/|\Omega|$ in (\ref{D:nint}) is sharp.

\begin{proof}

Let $v = h - M/|\Omega|$.  By (\ref{Lady_ineq}),
$$
\| v \|_{L^2(\Omega)}^2 \leqslant (\tfrac{3}{2})^{\tfrac{2}{3}}
\biggl(\int\limits_{\Omega} {v^2_{x} \,dx}\biggr)^{\tfrac{1}{3}}
\biggl(\int\limits_{\Omega} {|v| \,dx}\biggr)^{\tfrac{4}{3}}.
$$
\begin{multline*}
\mbox{Hence,} \quad \| h \|_{L^2(\Omega)}^2 \leqslant
(\tfrac{3}{2})^{\tfrac{2}{3}} \biggl(\int\limits_{\Omega} {h^2_{x}
\,dx}\biggr)^{\tfrac{1}{3}} \biggl(\int\limits_{\Omega} {\left|h -
\tfrac{M}{|\Omega|}\right|
\,dx}\biggr)^{\tfrac{4}{3}} +  \tfrac{M^2}{|\Omega|}\leqslant \quad \\
(\tfrac{3}{2})^{\tfrac{2}{3}} \biggl(\int\limits_{\Omega} {h^2_{x}
\,dx}\biggr)^{\tfrac{1}{3}} (2M)^{\tfrac{4}{3}} +
\tfrac{M^2}{|\Omega|}.
\end{multline*}
\end{proof}

Lemma \ref{Marinas_bound}
and the bound (\ref{C:d2'}) are used to prove $H^1$ control of the
generalized weak solution constructed in Theorem \ref{C:Th1}.

\begin{lemma}\label{G:Th1}
Let $h$ be the generalized solution of Theorem~\ref{C:Th1}. Then
\begin{equation}\label{G:1}
\tfrac{a_0}{4} \; \| h(\cdot,T_{loc}) \|_{H^1(\Omega)}^2
\leq
\mathcal{E}_0(0) + K T_{loc} + K_3
\end{equation}
where $\mathcal{E}_0(0)$ is defined in
(\ref{Energy}), $M = \int h_0$,
$K = |a_2 a_3| \| w' \|_\infty C$ and
$$
K_3 =
\begin{cases}
|a_2| \| w \|_\infty M & \mbox{ if } a_0 + a_1 \leq 0, \\
|a_2| \| w \|_\infty M +  M^2 \,\left( \frac{2
\sqrt{6}\,(a_0+a_1)^{3/2}}{3\sqrt{a_0}} + \tfrac{a_0 +
a_1}{2|\Omega|} \right) & \mbox{ otherwise}.
\end{cases}
$$
\end{lemma}
Note that if the evolution is missing either linear or nonlinear
advection ($a_2 = 0$ or $w' = 0$ or $a_3=0$) then Lemma \ref{G:Th1}
provides a uniform-in-time upper bound for
$\|h(\cdot,T_{loc})\|_{H^1}$.

For the equation (\ref{A:PukhEq}) which models the flow of a thin
film of liquid on the outside of a rotating cylinder one has $a_0 =
a_1 = \tfrac{\chi}{3}$, $a_2 = - \tfrac{\mu}{3}$, $a_3 = 1$, $w(x) =
\sin x$, and $|\Omega|=2\pi$. In this case, the $H^1$ bound
(\ref{G:1}) becomes
$$
\tfrac{\chi}{12} \|  h(.,T_{loc}) \|^2_{H^1(\Omega)} \leqslant
\mathcal{E}_0(0) + \tfrac{\mu}{3} C T_{loc} + \tfrac{\mu}{3} M + M^2
\left( \tfrac{8}{3} \sqrt{\chi}  + \tfrac{\chi}{6 \pi} \right)
$$
where $2 \mathcal{E}_0(0) = \int (\chi/3 \: (h_{0,x}^2 - h_0^2) +
2\mu/3 \: \sin(x) \: h_0 \;) dx$. The $H^1$ bound (\ref{G:1})
actually holds true for all times for which $h$ is strictly
positive. Recalling the definition (\ref{params}) of $\chi$, one
sees that the $H^1$ control is lost as $\chi \to 0$ (i.e. as
$\sigma/(\nu \rho R \omega) \to 0$), for example, in the zero
surface tension limit.

\begin{proof}
By (\ref{Energy}),
$$
\tfrac{a_0}{2} \int\limits_\Omega h_x^2(x,T) \; dx =
\mathcal{E}_0(T) + \tfrac{a_1}{2} \int\limits_\Omega h^2(x,T) \; dx
+ a_2 \int\limits_\Omega h(x,T) \, w(x) \; dx.
$$
The linear--in--time bound
(\ref{C:d2'}) on $\mathcal{E}_0(T_{loc})$
then implies
\begin{equation}
\tfrac{a_0}{2} \| h(\cdot,T_{loc}) \|_{H^1}^2
\leq
\mathcal{E}_0(0) +  K \,
 T_{loc} + \tfrac{a_0+a_1}{2} \int\limits_\Omega h^2
\; dx + |a_2| \| w \|_\infty M.
\label{general_bound}
\end{equation}
with $K = |a_2 a_3| \| w' \|_\infty C$.

\noindent
{\it \underline{Case 1: $a_0 + a_1 \leq 0 \;$}}
The third term on the right-hand side of
(\ref{general_bound}) is nonpositive and can be removed.  The desired
bound (\ref{G:1}) follows immediately. \\

\noindent {\it \underline{Case 2: $a_0 + a_1 > 0 \;$}} By Lemma
\ref{Marinas_bound} and Young's inequality
\begin{align} \notag
& \tfrac{a_0+a_1}{2} \int\limits_\Omega h^2 \; dx \leq
\tfrac{a_0+a_1}{2} \left( 6^{\tfrac{2}{3}}M^{\tfrac{4}{3}}
\biggl(\int\limits_{\Omega} {h^2_{x} \,dx}\biggr)^{\tfrac{1}{3}} +
\tfrac{M ^{2}}{|\Omega|} \right) \\
& \hspace{.2in} \leq \tfrac{a_0}{4} \int\limits_\Omega
h_x^2(x,T_{loc}) \; dx + M^2\,
\left(\tfrac{2\sqrt{6}(a_0+a_1)^{3/2}}{3\sqrt{a_0}} +
\tfrac{a_0+a_1}{2|\Omega|} \right). \label{young_marina}
\end{align}
Using this in (\ref{general_bound}), the desired bound (\ref{G:1})
follows immediately.
\end{proof}

This $H^1$ control in time of the generalized solution is now used
to extend the short--time existence result of Theorem \ref{C:Th1}
to a long--time existence result:
\begin{theorem} \label{global}
Let $T_g$ be an arbitrary positive finite number. The generalized
weak solution $h$ of Theorem \ref{C:Th1} can be continued in time
from $[0,T_{loc}]$ to $[0,T_g]$ in such a way that $h$ is also a
generalized weak solution and satisfies all the bounds of Theorem
\ref{C:Th1} (with $T_{loc}$ replaced by $T_g$).
\end{theorem}
Similarly, the short--time existence of strong solutions (see Theorem
\ref{C:Th1.3}) can be extended to a long--time existence.

\begin{proof}  To construct a weak solution up to time $T_g$, one applies
the local existence theory iteratively, taking the solution at the final time of
the current time interval as initial data for the next time interval.

Introduce the times
\begin{equation} \label{D:time_ints}
0 = T_0 < T_1 < T_2 < \dots < T_N < \dots \quad \mbox{where} \quad
T_N := \sum_{n=0}^{N-1} T_{n,loc}
\end{equation}
and $T_{n,loc}$ is the interval of existence (\ref{Tloc_is})
for a solution with initial data $h(\cdot,T_n)$:
\begin{equation} \label{D:Tloc_is}
T_{n,loc} := \tfrac{9}{40 c_6} \min\left\{ 1,   \left(
\int\limits_\Omega h_x^2(x,T_n) + 2 \tfrac{c_3}{a_0} G_0(h(x,T_n))
\; dx\right)^{-2} \right\}.
\end{equation}

The proof proceeds by contradiction.  Assume there
exists initial data $h_0$, satisfying the hypotheses of Theorem \ref{C:Th1},
that results in a
weak solution that cannot be extended arbitrarily in time:
$$
\sum_{k=0}^{\infty} T_{n,loc} = T^* < \infty
\quad
\Longrightarrow \quad \lim_{n \to \infty} T_{n,loc} = 0.
$$
From the definition (\ref{D:Tloc_is}) of $T_{n,loc}$, this implies
\begin{equation} \label{D:diverges}
\lim_{n \to \infty}
 \int\limits_\Omega (h_x^2(x,T_n)
+ 2 \tfrac{c_3}{a_0} G_0(h(x,T_n))) \; dx = \infty.
\end{equation}
%
$$
\mbox{By (\ref{G:1}) and  (\ref{C:d2'}),}\qquad\tfrac{a_0}{4}
\int\limits_\Omega h_x^2(x,T_n) \; dx \leq \mathcal{E}_0(T_{n-1}) +
K \, T_{n-1,loc} + K_3.
$$
$$
\mathcal{E}_0(T_{n-1}) \leq \mathcal{E}_0(T_{n-2}) + K \;
T_{n-2,loc}.
$$
Combining these,
$$
\tfrac{a_0}{4} \int\limits_\Omega h_x^2(x,T_n) \; dx
\leq
\mathcal{E}_0(T_{n-2})  + K  \left( T_{n-2,loc} + T_{n-1,loc} \right)
 + K_3 .
$$
\begin{equation}
\mbox{Continuing in this way,} \qquad \tfrac{a_0}{4}
\int\limits_\Omega h_x^2(x,T_n) \; dx \label{D:here} \leq
\mathcal{E}_0(0)   + K  \, T_n + K_3.
\end{equation}
By assumption, $T_n \to T^* < \infty$ as $n \to \infty$ hence $\int
h_x^2(x,T_n)\,dx$ remains bounded. Assumption (\ref{D:diverges})
then implies that $\int G_0(h(x,T_n))\,dx  \to \infty$ as $n \to
\infty$.

To continue, return to the approximate solutions $h_\eps$. By
(\ref{D:aa2}),
\begin{align} \label{D:2}
& \int\limits_\Omega G_\eps(h_\eps(x,T_{n,\eps})) \; dx
\leq \int\limits_\Omega G_\eps(h_\eps(x,T_{n-1,\eps})) \; dx \\
& \hspace{2in} \notag + c_5 \int\limits_{T_{n-1,\eps}}^{T_{n,\eps}}
\max \left\{ 1, \int\limits_\Omega h_{\eps,x}^2(x,T) \; dx \right\}
\; dT
\end{align}
Using (\ref{D:d2}), one proves the analogue of (\ref{G:1}) for all
$T \in [0,T_{loc,\eps}]$ and then the analogue of (\ref{D:here}) for
all $T \in [0,T_{n,\eps}]$.  Using this bound,
\begin{align}
\notag & \int\limits_{T_{n-1,\eps}}^{T_{n,\eps}} \int\limits_\Omega
h_{\eps,x}^2(x,T) \; dx dT
\leq
\tfrac{4}{a_0} \int\limits_{T_{n-1,\eps}}^{T_{n,\eps}} \mathcal{E}_\eps(0) + K \, T + K_3 \; dT \\
& \hspace{1in} \label{D:3} = \tfrac{4}{a_0} \left[
\mathcal{E}_\eps(0) + K_3 + \tfrac{K}{2} \left( T_{n-1,\eps} +
T_{n,\eps} \right) \right] \, T_{n-1,loc,\eps}.
\end{align}
Replacing $K_3$ by a larger value if necessary and using (\ref{D:3})
in (\ref{D:2}),
\begin{align} \label{D:4}
& \int\limits_\Omega G_\eps(h_\eps(x,T_{n,\eps})) \; dx \\
& \hspace{.5in} \notag \leq \int\limits_\Omega
G_\eps(h_\eps(x,T_{n-1,\eps})) \; dx + \left( \alpha + \beta
(T_{n-1,\eps} + T_{n,\eps}) \right) \, T_{n-1,loc,\eps}
\end{align}
for some $\alpha$ and $\beta$ which are fixed values that depend on
$|\Omega|$, the coefficients of the PDE,  and (possibly) on the
initial data $h_{0,\eps}$.
Taking $\eps_k \to 0$ in the sequence $\{ \eps_k \}$ that was used
to construct $h$ yields
\begin{equation}
\label{D:4a}
\int\limits_\Omega G_0(h(x,T_{n}))dx
\leq \int\limits_\Omega G_0(h(x,T_{n-1}))dx
+ \left( \alpha + \beta (T_{n-1} + T_{n})\right) \, T_{n-1,loc}.
\end{equation}
Applying (\ref{D:4a}) iteratively and using that $T_k < T^*$,
\begin{equation}
 \int\limits_\Omega G_0(h(x,T_n)) \; dx
 \leq
 \int\limits_\Omega G_0(h_0(x)) \; dx
 + \left( \alpha + \beta \, 2 \, T^*\right) T_n.
\end{equation}
Hence $\int G_0(h(x,T_n))dx < \infty$ as $n \to \infty$, finishing
the proof.
\end{proof}

Under certain conditions, a bound closely related to (\ref{G:1})
implies that if the solution of Theorem \ref{C:Th1} is initially
constant then it will remain constant for all time:
\begin{theorem}\label{constancy}
Assume the coefficients $a_1$ and $a_2$ in (\ref{A:mainEq}) satisfy
$a_1 \geq 0$, $a_2=0$ and $| \Omega | < 4 a_0/|a_1|$.  If the
initial data is constant, $h_0 \equiv C > 0$, then the solution of
Theorem \ref{C:Th1} satisfies $h(x,t) = C$ for all $x \in
\bar{\Omega}$ and all $t > 0$.
\end{theorem}
The hypotheses of Theorem \ref{constancy} correspond to the
equation is long--wave unstable ($a_1 > 0$), there is no nonlinear
advection ($a_2 = 0$), and the domain is not ``too large''.

\begin{proof}
Consider the approximate solution $h_\eps$.
The definition of $\mathcal{E}_\eps(T)$
combined with the linear-in-time bound  (\ref{D:d2}) implies
\begin{equation} \label{gen_bound2}
\tfrac{a_0}{2} \int\limits_\Omega h_{\eps,x}^2(x,T) \; dx \leq
\mathcal{E}_\eps(0) + K  \, T  + \tfrac{|a_1|}{2} \int\limits_\Omega h_\eps^2 \;
dx
+ | a_2 | \| w \|_\infty M_\eps
\end{equation}
where $M_\eps = \int h_{0,\eps}\,dx$. Applying Poincar\'e's
inequality (\ref{Poincare}) to $v_\eps = h_\eps - M_\eps/|\Omega|$
and using $\int h_\eps^2 \, dx= \int v_\eps^2 \, dx +
M_\eps^2/|\Omega|$ yields
$$
\left(
\tfrac{a_0}{2} - \tfrac{|a_1| \, |\Omega|^2}{8}
\right)
\int\limits_\Omega h_{\eps,x}^2(x,t) \; dx
\leq
\mathcal{E}_\eps(0)  + K \, T_{\eps,loc}  + \tfrac{|a_1| M_\eps^2}{2 |\Omega|}
+ | a_2 | \| w \|_\infty M_\eps.
$$
If $h_{0,\eps} \equiv C_\eps = C + \eps^\theta$ and $a_2 = 0$ (hence $K = 0$) this becomes
$$
\left(
\tfrac{a_0}{2} - \tfrac{|a_1| |\Omega|^2}{8}
\right)
\int\limits_\Omega h_{\eps,x}^2(x,T) \; dx
\leq (a_1 - |a_1|) \tfrac{C^2 |\Omega|}{2}.
$$
If $a_1 \geq 0$ and $|\Omega| < 4 a_0/a_1$ then $\int
h_{\eps,x}^2(x,T) \; dx = 0$ for all $T \in [0,T_{\eps,loc}]$ and
that this, combined with the continuity in space and time of
$h_\eps$, implies that $h_\eps \equiv C_\eps$ on $Q_{T_{\eps,loc}}$.
Taking the sequence $\{ \eps_k \}$ that yields convergence to the
solution $h$ of Theorem \ref{C:Th1}, $h \equiv C$ on $Q_{T_{loc}}$.

\end{proof}

\section{Strong positivity of solutions} \label{F}

\begin{proposition}\label{F:LemPos}
Assume the initial data $h_0$ satisfies
$h_0(x) > 0 $ for all
$x \in \omega \subseteq \Omega$ where $\omega$ is an open
interval.
Then the weak solution $h$ from Theorem~\ref{C:Th1}
satisfies:

1) $h(x,T) > 0$ for almost every $x \in \omega$, for all $T \in
[0,T_{loc}]$;

2) $h(x,T) > 0$ for all $x \in \omega$,  for almost every $T \in
[0,T_{loc}]$.
\end{proposition}

The proof of Proposition~\ref{F:LemPos} depends on a local version
of the a priori bound (\ref{BF_entropy}) of Lemma \ref{MainAE}:

\begin{lemma} \label{F:local_BF}  Let $\omega \subseteq \Omega$ be an open interval
and
$\zeta \in C^2(\bar{\Omega})$ such that
$\zeta > 0$ on $\omega$,
$\text{supp}\,\zeta = \overline{\omega}$,
and $(\zeta^4)^{\prime} = 0$ on
$\partial\Omega$. If $\omega = \Omega$, choose $\zeta$ such
that $ \zeta(-a) = \zeta(a)
> 0$. Let $\xi := \zeta^4$.

If the initial data $h_0$ and the
time $T_{loc}$ are as in Theorem \ref{C:Th1}
then for all $T \in [0,T_{loc}]$ the weak solution $h$
from Theorem \ref{C:Th1}
satisfies
\begin{equation} \label{local_BF_finite}
\int\limits_\Omega \xi(x) \; \tfrac{1}{h(x,T)} \; dx < \infty
\end{equation}

\end{lemma}
The proof of Lemma \ref{F:local_BF} is given in Appendix \ref{A_priori_proofs}.
The proof of Proposition \ref{F:LemPos} is essentially a combination of the proofs
of Corollary 4.5 and Theorem 6.1 in \cite{B8} and is provided here for the reader's
convenience.
\begin{proof}[Proof of Proposition~\ref{F:LemPos}]
Choose
the localizing function
$\zeta(x)$ to satisfy the hypotheses of
Lemma \ref{F:local_BF}.  Hence, (\ref{local_BF_finite}) holds
for every $T \in [0,T_{loc}]$.

First, we prove
$h(x,T) > 0$ for almost every $x \in \omega$, for all $T \in [0,T_{loc}]$.
Assume not.  Then there is a time $T \in [0,T_{loc}]$ such that the set
$\{ x \; | \; h(x,T) = 0 \} \cap \omega$ has positive measure.  Then
$$
\infty > \int\limits_\Omega \xi(x) \tfrac{1}{h(x,T)} \; dx
\geq
\int\limits_{ \{h(\cdot,T) = 0 \} \cap \omega } \xi(x) \tfrac{1}{h(x,T)} \; dx
= \infty.
$$
This contradiction implies
there can be no time at which $h$ vanishes on a set of positive measure
in $\omega$, as desired.

Now, we prove
$h(x,T) > 0$ for all $x \in \omega$,  for almost every $T \in [0,T_{loc}]$.
By (\ref{Linf_H2}), $h_{xx}(\cdot,T) \in L^2(\Omega)$ for almost all $T \in [0,T_{loc}]$
hence $h(\cdot,T) \in C^{3/2}(\Omega)$ for almost all $T \in [0,T_{loc}]$.
Assume $T_0$ is such that
$h(\cdot,T_0) \in C^{3/2}(\Omega)$ and $h(x_0,T_0) = 0$ at some $x_0 \in \omega$.
Then there is a $L$ such that
$$
h(x,T_0) = |h(x,T_0)-h(x_0,T_0)| \leq L | x-x_0|^{3/2}.
$$
Hence
$$
\infty > \int\limits_\Omega \xi(x) \tfrac{1}{h(x,T_0)} \; dx
\geq
\tfrac{1}{L} \int\limits_\Omega \xi(x) |x - x_0|^{-3/2} \; dx = \infty.
$$
This contradiction implies there
can be no point $x_0$ such that $h(x_0,T_0) = 0$, as desired.  Note that
we used $\xi > 0$ on $\omega$ and $x_0 \in \omega$ to conclude that the
integral diverges.

\end{proof}

We close our discussion by illustrations of positivity and long time
existence via numerical simulations of the initial value problem for
different regimes of the PDE.

Figure \ref{no_advect} considers the PDE with no advection, $h_t +
(h^3 ( h_{xxx} + 16\, h_x))_x = 0$.  The PDE is translation
invariant in $x$ and constant steady states are linearly unstable.
As a result, any non-constant behaviour observed in a solution
starting from constant initial data would be due to growth of
round-off error. For this reason, non-constant initial data is
chosen: $h_0(x) = 0.3 + 0.02 \, \cos(x) + 0.02 \, \cos(2 x)$. The
$L^2$ and $H^1$ norms of the resulting solution appear to be
converging to limiting values as time passes and long-time limit of
the solution appears to be four steady-state droplets of the form $a
\cos(4 x + \phi) + b$ for appropriate values of $a$, $\phi$, and
$b$. Like the PDE, the simulation shown respects the symmetry about
$x=0$ of the initial data.

Figure \ref{no_linear_advect} shows the evolution from constant
initial data for the PDE with nonlinear advection but no linear
advection: $h_t + (h^3 ( h_{xxx} + 16\, h_x - 8 \cos(x)))_x = 0$.
The long-time limit appears to be a steady state which is zero (or
nearly zero on $[-\pi,0]$ ) with a droplet supported within
$(0,\pi)$ and centred roughly about the mid point ($x= \pi/2$).

Finally, Figure \ref{with_advect} shows the evolution resulting from
the same constant initial data for the PDE with both linear and
nonlinear advection: $h_t + (h^3 ( h_{xxx} + 16\, h_x - 8
\cos(x)))_x + 3 h_x= 0$. The long-time limit appears to be a
strictly positive steady state.

We close by noting that the PDE considered in Figure
\ref{with_advect} corresponds to coefficient $a_3 = 3$ in the PDE
(\ref{A:mainEq}).  As we increase the value of $a_3$ we find there
appears to be a critical value past which the solution appears to
converge to a time-periodic behaviour rather than a steady state.

\begin{figure}
\begin{center}
\includegraphics[width= 6cm] {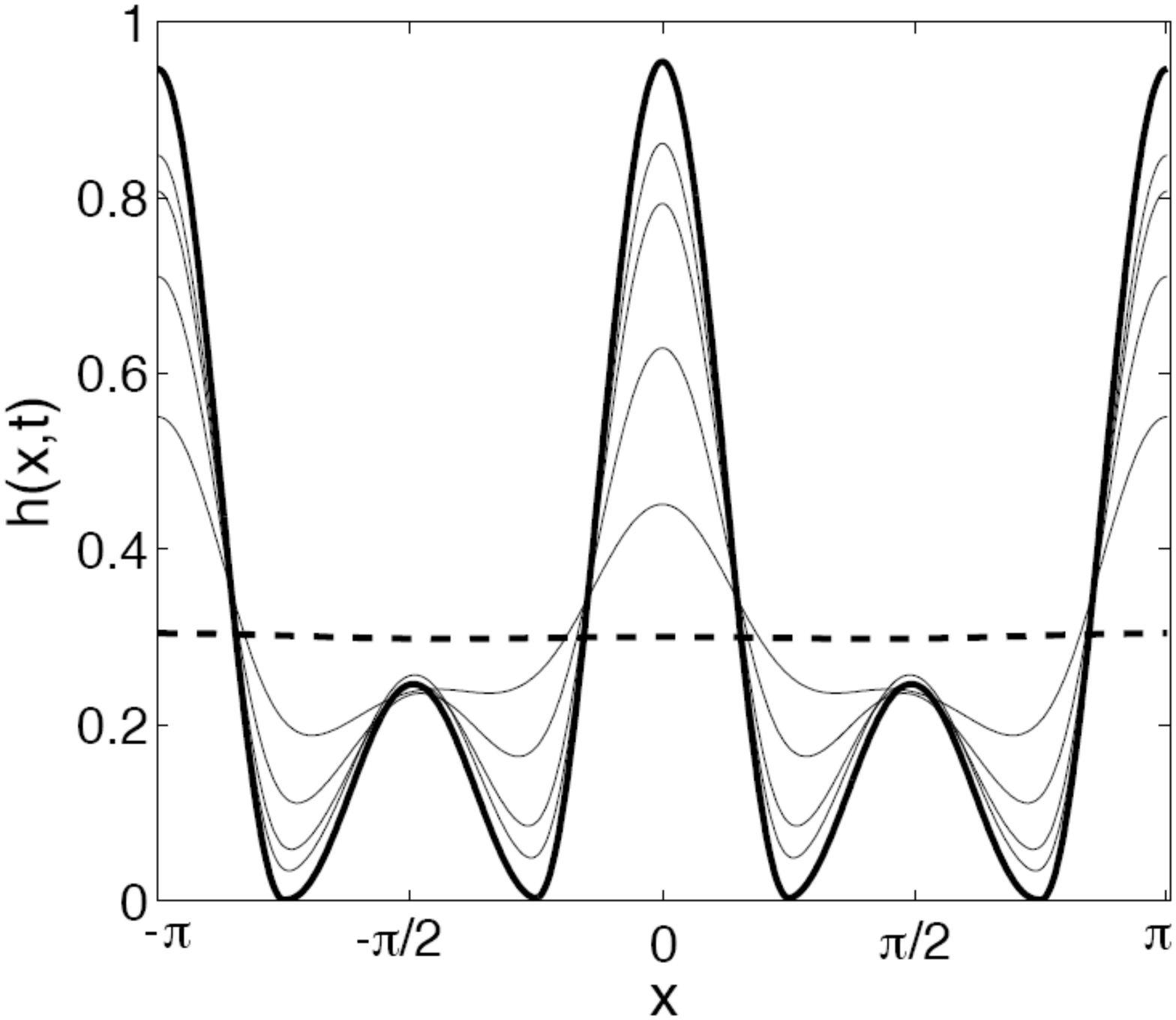}
\includegraphics[width= 6cm] {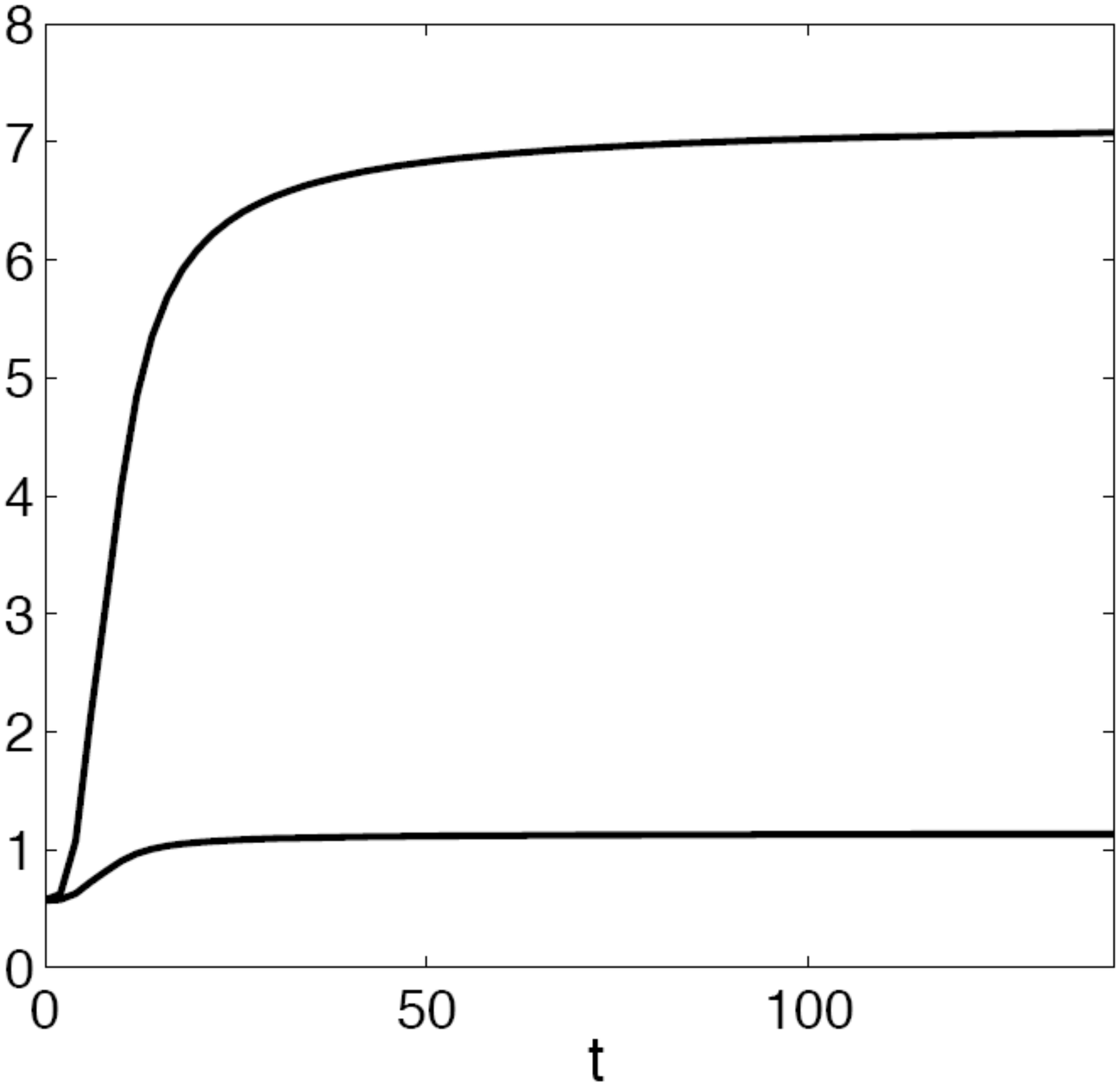}
\end{center}
\vspace{-.4in} \caption{\label{no_advect} The evolution equation
with no linear or nonlinear advection, $h_t + (h^3 ( h_{xxx} + 16\,
h_x))_x = 0$, corresponding to $a_0 = 1$, $a_1 = 16$, and $a_2 = a_3
= 0$.  The initial data is $h_0(x) = 0.3 + 0.02 \, \cos(x) + 0.02 \,
\cos(2 x)$.  Left plot: the solution at times $t=0$ (dashed line),
$t = 12, 12.5, 13, 15$ (solid lines), and $t=140$ (heavy line).
Right plot: the $L^2$ and $H^1$ norms plotted as a function of time.
}
\end{figure}

\begin{figure}
\begin{center}
\includegraphics[height=4.7cm] {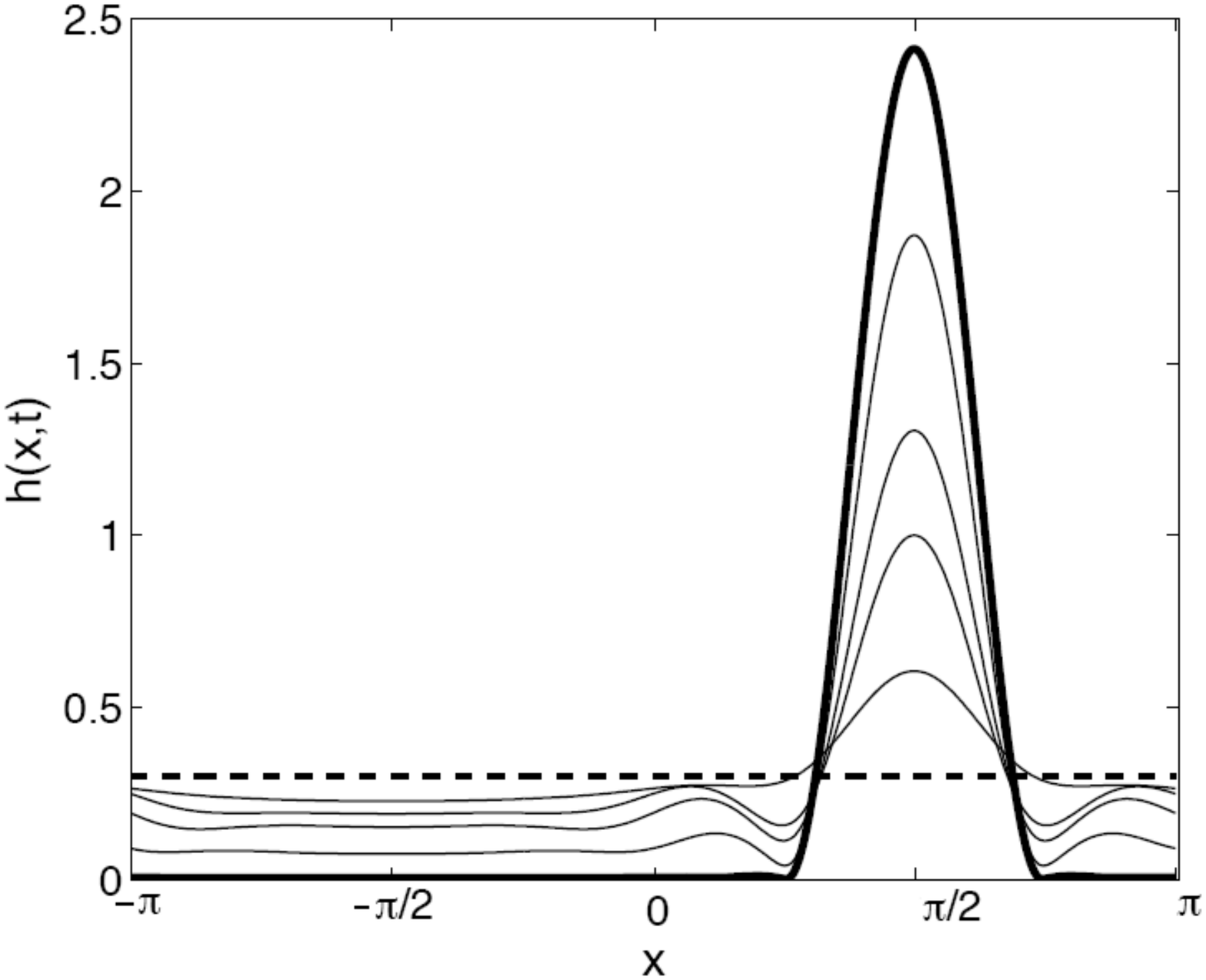}
\includegraphics[height=4.7cm] {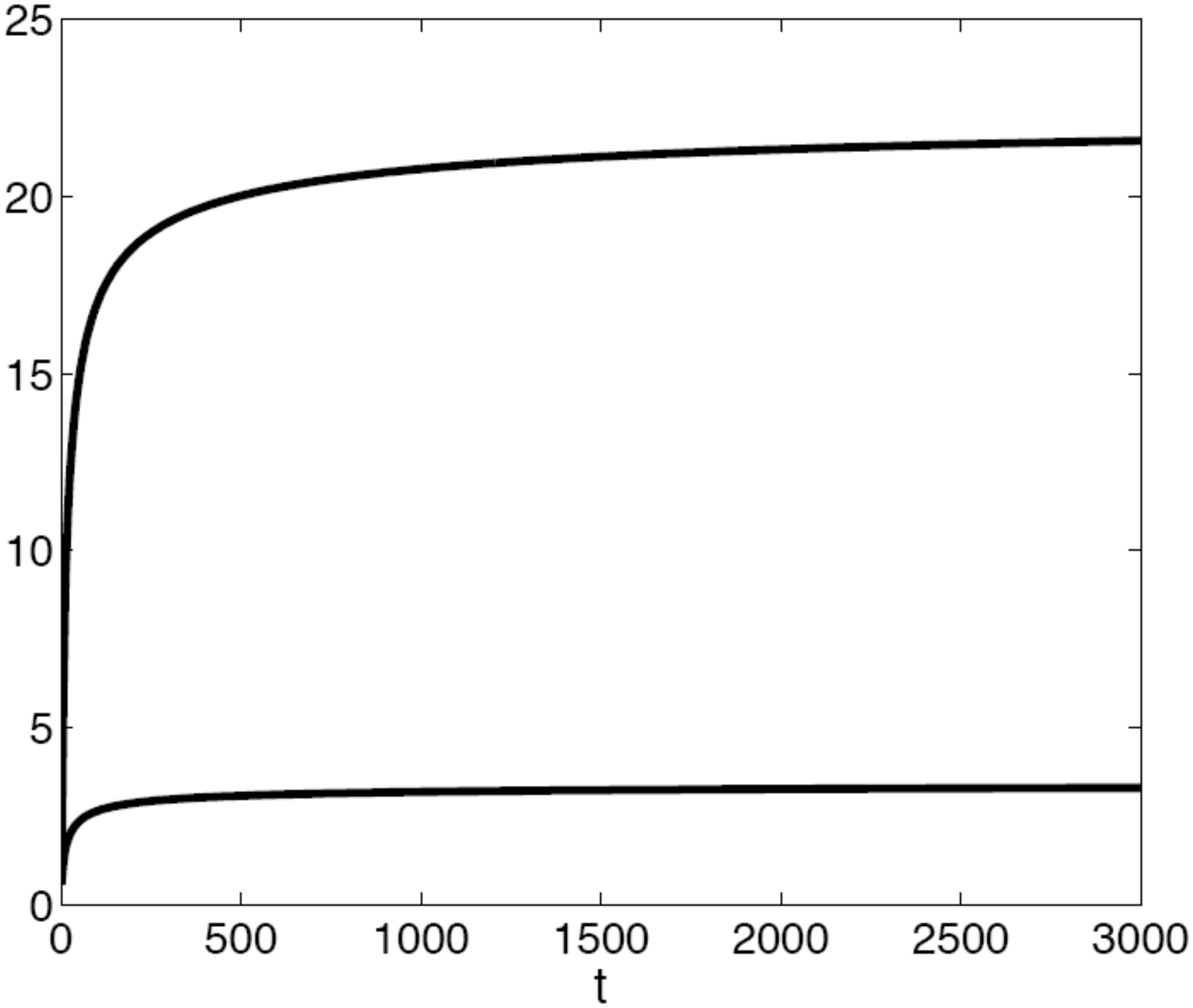}
\end{center}
\vspace{-.4in} \caption{\label{no_linear_advect} The evolution
equation with nonlinear advection but no linear advection, $h_t +
(h^3 ( h_{xxx} + 16\, h_x - 8 \cos(x)))_x = 0$, corresponding to
$a_0 = 1$, $a_1 = 16$, $a_2 = 8$, and $a_3 = 0$.  The initial data
is $h_0(x) = 0.3$. Left plot: the solution at  times $t=0$ (dashed
line), $t = 0.5, 1, 2, 10$ (solid lines), and $t=3000$ (heavy line).
Right plot: the $L^2$ and $H^1$ norms plotted as a function of time.
}
\end{figure}

\begin{figure}
\begin{center}
\includegraphics[height=4.7cm] {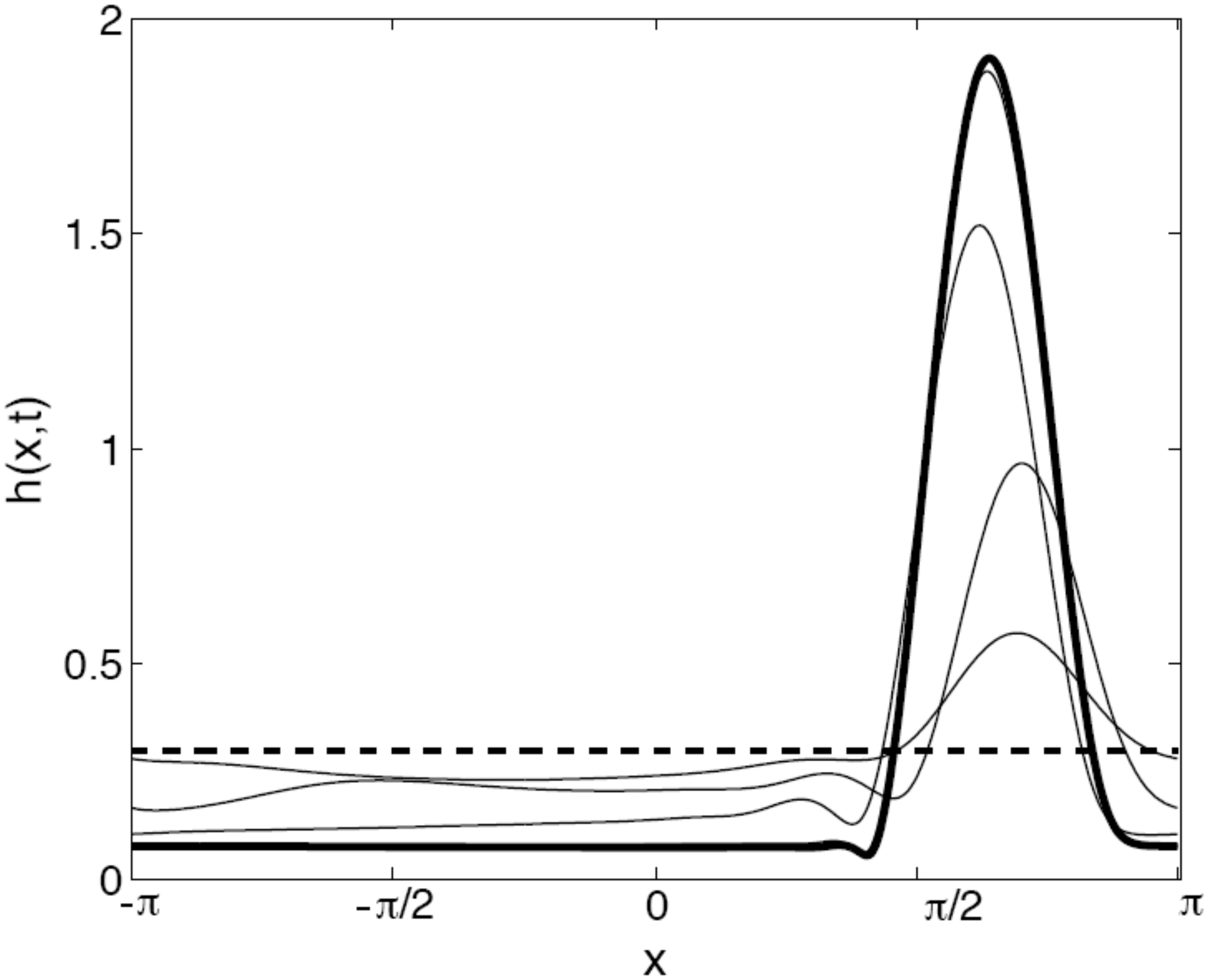}
\includegraphics[height=4.7cm] {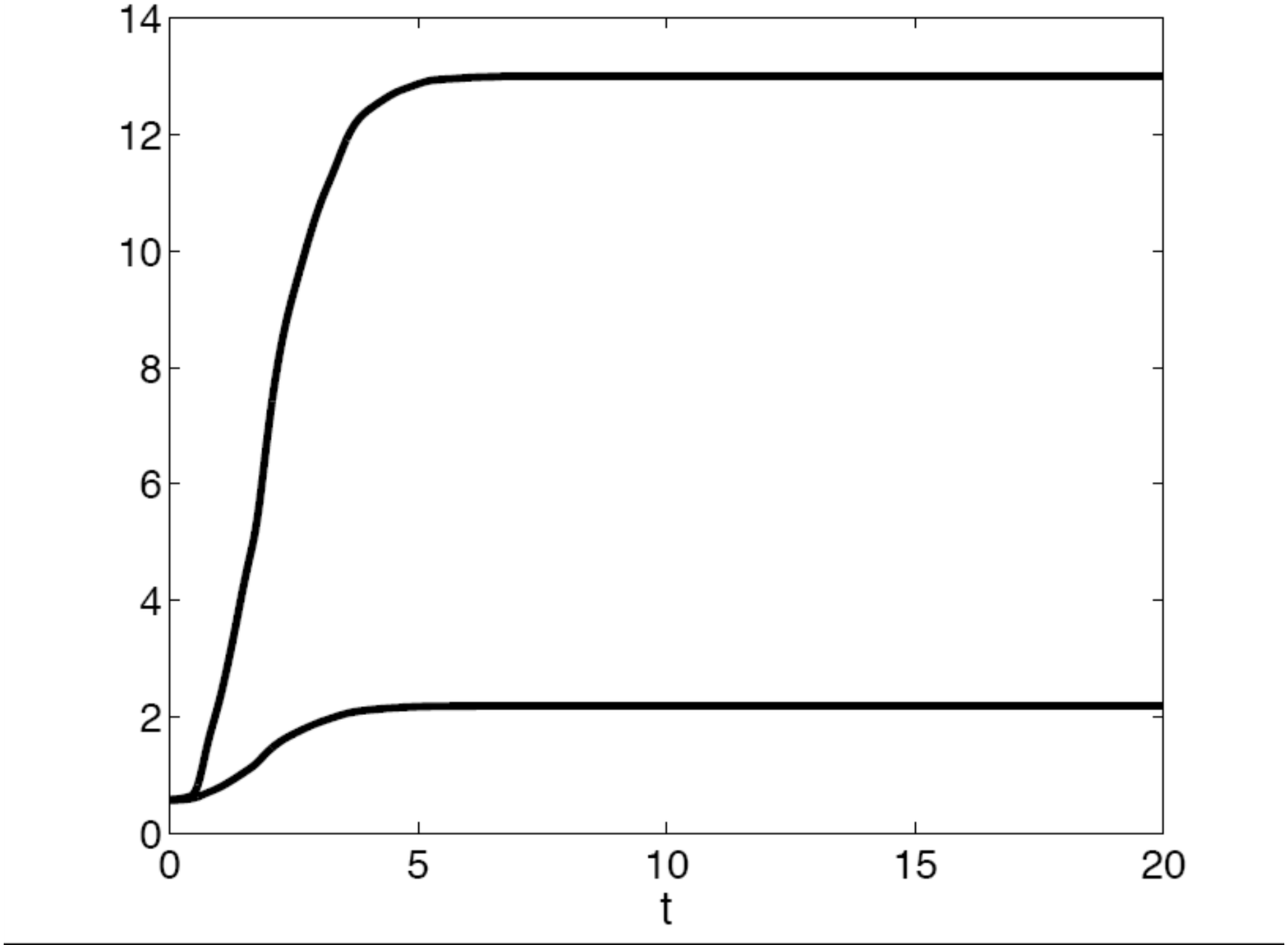}
\end{center}
\vspace{-.4in} \caption{\label{with_advect} The evolution equation
with both linear and nonlinear advection, $h_t + (h^3 ( h_{xxx} +
16\, h_x - 8 \cos(x)))_x + 3 h_x= 0$, corresponding to $a_0 = 1$,
$a_1 = 16$, $a_2 = 8$, and $a_3 = 3$.  The initial data is $h_0(x) =
0.3$. Left plot: the solution at times $t=0$ (dashed line), $t =0.5,
1, 2, 4$ (solid lines), and $t=20$ (heavy line).  Right plot: the
$L^2$ and $H^1$ norms plotted as a function of time. }
\end{figure}

\appendix

\section{Proofs of A Priori Estimates}
\label{A_priori_proofs}

The first observation is that the periodic boundary conditions imply
that classical solutions of equation (\ref{D:1r'}) conserve mass:
\begin{equation}\label{D:mass}
\int\limits_{\Omega} {h_{\delta \eps}(x,t)\,dx} =  \int\limits_{\Omega}
{h_{0,\delta \varepsilon}(x) \,dx} = M_{\delta \varepsilon } < \infty \text{ for all } t
> 0.
\end{equation}
Further, (\ref{D:inreg}) implies $M_{\delta \eps} \to M = \int h_0$ as $\eps, \delta \to 0$.
The initial data in this article have $M > 0$, hence $M_{\delta \eps} > 0$
for $\delta$ and $\eps$ sufficiently small.

Also, we will relate the $L^p$ norm of $h$ to the $L^p$ norm of its zero-mean part as follows:
$$
|h(x)| \leq \left| h(x)-\tfrac{M}{|\Omega|} \right| +
\tfrac{M}{|\Omega|} \Longrightarrow \| h \|_p^p \leq 2^{p-1} \; \| v
\|_p^p + \left(\tfrac{2}{  |\Omega| } \right)^{p-1} \; M^{p}
$$
where $v := h-M/|\Omega|$ and we have assumed that $M \geq 0$. We
will use the Poincar\'{e} inequality which holds for  any zero-mean
function in $H^1(\Omega)$
\begin{equation} \label{Poincare}
\| v \|_p^p \leq b_1 \|v_x \|_p^p \qquad 1 \leq p < \infty
\end{equation}
with $b_1 = |\Omega|^p/(p \: 2^{p-1})$.

Also used will be an interpolation inequality \cite[Th. 2.2, p. 62]{Lady} for functions of zero mean in $H^1(\Omega)$:
\begin{equation} \label{Lady_ineq}
\| v \|_p^p \leq b_2 \, \| v_x \|_2^{ap} \; \| v \|_r^{(1-a)p}
\end{equation}
where $r \geq 1$, $p \geq r$,
$$
a = \tfrac{1/r - 1/p}{1/r + 1/2}, \qquad
b_2 = \left(1 + r/2 \right)^{a p}.
$$
It follows that for any zero-mean function $v$ in $H^1(\Omega)$
\begin{equation} \label{LpH1}
\| v \|_p^p  \leq b_3 \| v_x \|_2^p,
\quad
\Longrightarrow
\quad
\| h \|_p^p \leq b_4 \| h_x \|_2^p + b_5 M_{\delta \eps}^p
\end{equation}
where
$$
b_3 =
\begin{cases}
b_1 \; | \Omega |^{(2-p)/p} & \mbox{if} \quad 1 \leq p \leq 2 \\
b_1^{(p+2)/2} \; b_2 & \mbox{if} \quad  2 < p < \infty
\end{cases}, \quad
b_4 = 2^{p-1} \, b_3, \quad b_5 = \left( \tfrac{2}{|\Omega|} \right)^{p-1}
$$
To see that (\ref{LpH1}) holds, consider two cases.  If $1 \leq p < 2$, then by (\ref{Poincare}), $\| v \|_p$ is controlled by $\| v_x \|_p$.  By the H\"older inequality, $\| v_x \|_p$ is then controlled by $\| v_x \|_2$.  If $p > 2$ then by (\ref{Lady_ineq}), $\|v \|_p$ is controlled by $\|v_x\|_2^a \|v \|_2^{1-a}$ where $a = 1/2 - 1/p$.  By  the Poincar\'e inequality, $\| v \|_2^{1-a}$ is controlled by $\| v_x \|_2^{1-a}$.

\begin{proof}[Proof of Lemma~\ref{MainAE}]
In the following, we denote the classical solution $h_{\delta \eps}$ by $h$ whenever
there is no chance of confusion.

To prove the bound (\ref{D:a13''}) one starts by multiplying
(\ref{D:1r'}) by $- h_{xx}$, integrating over $Q_T$, and using the
periodic boundary conditions (\ref{D:2r'}) yields
\begin{align}
\label{D:a1}
& \tfrac{1}{2} \int\limits_{\Omega} {h_x^2(x,T) \,dx} + a_0
\iint\limits_{Q_T} {f_{\delta \varepsilon}(h) h^2_{xxx} \,dx dt} \\
\notag & \hspace{.2in} = \tfrac{1}{2}\int\limits_{\Omega}
{{h_{0,\delta \varepsilon,x}}^2(x) \,dx} - a_1 \iint\limits_{Q_T}
{f_{\varepsilon}(h) h_x h_{xxx}\,dx dt}+
\delta a_1 \iint\limits_{Q_T}h_{xx}^2 \; dx dt \\
\notag & \hspace{.4in} - a_2 \iint\limits_{Q_T} {f_{\delta
\varepsilon}(h)w'  h_{xxx}\,dx dt} - \delta a_2 \iint\limits_{Q_T}
w^{\prime} \, h_{xxx} \; dx dt.
\end{align}

By Cauchy and Young inequalities, due to
(\ref{Poincare})--(\ref{LpH1}), it follows from (\ref{D:a1}) that
\begin{align}
\label{D:a11} & \tfrac{1}{2} \int\limits_{\Omega} {h_x^2(x,T)
\,dx} + \tfrac{a_0}{2}\iint\limits_{Q_T} {f_{\delta \varepsilon}(h) h^2_{xxx}\,dx dt} \\
& \notag \hspace{.2in} \!\!\!\!\! \leq
\tfrac{1}{2}\int\limits_{\Omega} {{h_{0,\delta \varepsilon,x}}^2
\,dx} + c_3 \iint\limits_{Q_T} {h^2_{xx}\,dx dt} + c_4
\int\limits_0^T \max\left\{ 1 , \left( \int\limits_\Omega h_x^2 \;
dx \right)^3 \right\} \; dt
\end{align}
where  $c_1 = b_2^2/8 + b_4/2$,  $c_2 = M_{\delta \eps}^6 \; b_5/2$,
$c_3 = \tfrac{a_1^2}{2 a_0} + \delta |a_1|$, \\ $ c_4 =
\tfrac{a_1^2}{a_0} c_1 + \tfrac{a_2^2}{a_0} \| w' \|_\infty^2 b_4 +
\tfrac{a_1^2}{a_0}c_2 + \tfrac{a_2^2}{a_0} \| w' \|_\infty^2 b_5
M_{\delta \eps}^3 + \delta \tfrac{a_2^2}{a_0} \| w' \|_2^2$.

Now, multiplying (\ref{D:1r'}) by $G'_{\delta \varepsilon}(h)$,
integrating over $Q_T$, and using the periodic boundary conditions
(\ref{D:2r'}), we obtain
\begin{multline}\label{D:aa1}
\int\limits_{\Omega} {G_{\delta \varepsilon}(h(x,T))\,dx} + a_0
\iint\limits_{Q_T} {h^2_{xx}\,dx dt} = \int\limits_{\Omega}
{G_{\delta \varepsilon}(h_{0, \delta \varepsilon})dx} +
a_1 \iint\limits_{Q_T} {h^2_{x}dx dt} \\
-a_3 \iint\limits_{Q_T} {(G_{\delta \varepsilon}(h))_{x} \,dx dt} + a_2
\iint\limits_{Q_T} {w' h_{x}\,dx dt}.
\end{multline}
By the periodic boundary conditions, we deduce
\begin{align}\notag
& \int\limits_{\Omega} {G_{\delta \varepsilon}(h(x,T))\,dx} + a_0
\iint\limits_{Q_T} {h^2_{xx}\,dx dt} \\
& \label{D:aa2} \hspace{.3in} \leqslant \int\limits_{\Omega}
{G_{\delta \varepsilon}(h_{0, \delta \varepsilon})\,dx} + c_5
\int\limits_0^T \max\left\{ 1, \int\limits_\Omega h_x^2(x,t) \; dx
\right\} \; dt,
\end{align}
where $ c_5 = |a_1| + |a_2| \| w' \|_2 $. Further, from
(\ref{D:a11}) and (\ref{D:aa2}) we find
\begin{align}\label{D:aa3} \notag
& \int\limits_{\Omega} {h_x^2 \,dx} + \tfrac{2 c_3}{a_0}
\int\limits_\Omega G_{\delta \eps}(h) \; dx + a_0 \iint\limits_{Q_T}
{f_{\delta \varepsilon}(h) h^2_{xxx}  \,dx dt} \leq  \int\limits_\Omega {h_{0,\delta \eps,x}}^2dx \\
& \hspace{.1in}  + \tfrac{2 c_3}{a_0} \int\limits_\Omega G_{\delta
\eps}(h_{0,\delta \eps})dx + c_6 \int\limits_0^T \max\left\{ 1,
\left( \int\limits_\Omega h_x^2(x,t) \; dx \right)^3 \right\}dt
\end{align}
where $c_6 = 2 c_3 c_5/a_0 + 2 c_4$. Applying the nonlinear
Gr\"onwall lemma \cite{Bihari} to
$$
v(T) \leq v(0) + c_6 \int\limits_0^T \max\{1,v^3(t)\} \; dt$$ with
$v(t) = \int (h_x^2(x,t) + 2 c_3/a_0 \: G_{\delta \eps}(h(x,t)))
\; dx$ yields
\begin{align} \label{H1_control}
& \int\limits_\Omega h_x^2(x,t) + 2 \tfrac{c_3}{a_0} G_{\delta \eps}(h(x,t)) \; dx \\
& \leq \sqrt{2} \max\left\{ 1, \int\limits_\Omega ({h_{0,\delta
\eps,x}}^2(x) + 2 \tfrac{c_3}{a_0} G_{\delta \eps}(h_{0,\delta
\eps}(x))) \; dx \right\} = K_{\delta \eps} < \infty \notag
\end{align}
for all $t \in [0,T_{\delta \eps,loc}]$ where
\begin{equation} \label{Tloc_eps_is}
T_{\delta \eps,loc} := \tfrac{1}{4 c_6} \min\left\{ 1,   \left(
\int\limits_\Omega ({h_{0,\delta \eps,x}}^2(x) + 2 \tfrac{c_3}{a_0}
G_{\delta \eps}(h_{0,\delta \eps}(x)))\,dx \right)^{-2} \right\}.
\end{equation}
Using the $\delta \to 0, \eps \to 0$ convergence of the initial data
and the choice of $\theta \in (0,2/5)$
(see (\ref{D:inreg})) as well as the assumption that the initial data $h_0$ has
finite entropy (\ref{C:inval}),
the times $T_{\delta \eps, loc}$ converge to a positive limit and
the upper bound $K$ in
(\ref{H1_control}) can be taken finite and independent of $\delta$ and $\epsilon$
for $\delta$ and $\eps$ sufficiently small.  (We refer the reader to the end
of the proof of Lemma \ref{F:local_BF} in this Appendix for a fuller explanation
of a similar case.)
Therefore there
exists $\delta_0>0$ and $\eps_0>0$ and $K$ such that
the bound (\ref{H1_control}) holds for all $0 \leq \delta < \delta_0$ and
$0 < \eps < \eps_0$
with $K$ replacing $K_{\delta \eps}$ and for all
\begin{equation} \label{Tloc_is}
0 \leq t \leq T_{loc} := \tfrac{9}{10} \lim_{\eps \to 0, \delta \to 0} T_{\delta \eps,loc}.
\end{equation}

Using the uniform bound on $\int h_x^2$ that (\ref{H1_control})
provides, one can find a uniform-in-$\delta$-and-$\eps$ bound for the
right-hand-side of (\ref{D:aa3}) yielding the desired a priori bound
(\ref{D:a13''}).  Similarly, one can find a
uniform-in-$\delta$-and-$\eps$ bound for the right-hand-side of
(\ref{D:aa2}) yielding the desired a priori bound (\ref{BF_entropy}).

To prove the bound (\ref{D:d2}), multiply (\ref{D:1r'}) by $- a_0
h_{xx} - a_1 h - a_2 w$, integrate over $Q_T$, integrate by parts,
use the periodic boundary conditions  (\ref{D:2r'}), and use the
mass conservation (see (\ref{D:mass})) to find
\begin{align}
& \notag \mathcal{E}_{\delta \varepsilon}(T) + \iint\limits_{Q_T}
{f_{\delta \varepsilon}(h) (a_0 h_{xxx} + a_1 h_x + a_2 w'(x))^2
\,dx dt} \\
& \hspace{.4in}  \label{D:d0} \leq \mathcal{E}_{\delta
\varepsilon}(0) + | a_2 a_3| \| w' \|_\infty \left( |\Omega|^2
\sqrt{K_1} +  2 M \right) T.
\end{align}
Hence the desired bound (\ref{D:d2}) is obtained if the constant
$$
K_3 =  | a_2 a_3| \| w' \|_\infty ( |\Omega|^2 \sqrt{K_1} +  2 M ).
$$
The time $T_{loc}$ and the constants $K_1$, $K_2$, and $K_3$ are
determined by $\delta_0$, $\eps_0$, $a_0$, $a_1$, $a_2$, $w'$,
$|\Omega |$, and $h_0$.
\end{proof}

\begin{proof}[Proof of Lemma~\ref{MainAE2}]
In the following, we denote the positive, classical solution $h_{\eps}$ by $h$ whenever
there is no chance of confusion.

Multiplying (\ref{D:1r'}) by $(G^{(\alpha)}_{\varepsilon} (h))'$,
integrating over $Q_T$, taking $\delta \to 0$,
and using the periodic boundary conditions
(\ref{D:2r'}), yields
\begin{align}
& \label{E:b1}
\int\limits_{\Omega} {G^{(\alpha)}_{\varepsilon}(h(x,T))\,dx} + a_0
\iint\limits_{Q_T} {h^{\alpha}h^2_{xx}\,dx dt}  +
a_0 \tfrac{\alpha(1 - \alpha)}{3} \iint\limits_{Q_T} {h^{\alpha -2
}h^4_{x}\,dx dt} \\
& =
\int\limits_{\Omega}
{G^{(\alpha)}_{\varepsilon}(h_{0\varepsilon})\,dx} + a_1
\iint\limits_{Q_T} {h^{\alpha} h^2_{x}\,dx dt}
- \tfrac{a_2}{\alpha+1} \iint\limits_{Q_T} {h^{\alpha+1} w'' \,dx dt}.
\notag
\end{align}

\textbf{Case 1: $0 < \alpha < 1$.} The coefficient
multiplying $\iint h^{\alpha-2} h_x^4$ in (\ref{E:b1}) is positive and
can therefore be used to control the term
$\iint h^\alpha h_x^2$ on the
right--hand side of (\ref{E:b1}).  Specifically, using the Cauchy-Schwartz
inequality and the Cauchy inequality,
\begin{equation} \label{E:b2}
a_1\iint\limits_{Q_T} {h^{\alpha} h^2_{x}}\; dx dt \leqslant
\tfrac{a_0 \alpha(1 - \alpha)}{6}\iint\limits_{Q_T} {h^{\alpha-2}
h^4_{x}}\; dx dt + \tfrac{3a_1^2}{2a_0\alpha(1 -
\alpha)}\iint\limits_{Q_T} {h^{\alpha + 2} }\;dx dt.
\end{equation}
Using the bound (\ref{E:b2}) in (\ref{E:b1}) yields
\begin{align}
& \label{E:b3}
\int\limits_{\Omega} {G^{(\alpha)}_{\varepsilon}(h(x,T))\,dx} + a_0
\iint\limits_{Q_T} {h^{\alpha}h^2_{xx}\,dx dt}  +
a_0 \tfrac{\alpha(1 - \alpha)}{6} \iint\limits_{Q_T} {h^{\alpha -2
}h^4_{x}\,dx dt} \\
& \leq \int\limits_{\Omega}
{G^{(\alpha)}_{\varepsilon}(h_{0\varepsilon})\,dx} + \tfrac{3
a_1^2}{2 a_0 \alpha (1-\alpha)} \iint\limits_{Q_T} {h^{\alpha+2}\;
dx dt } + \tfrac{|a_2| \| w'' \|_\infty}{\alpha+1}
\iint\limits_{Q_T} {h^{\alpha+1}dxdt }. \notag
\end{align}
By  (\ref{LpH1}),
\begin{align}
&\notag \int\limits_{\Omega}
{G^{(\alpha)}_{\varepsilon}(h(x,T))\,dx} + a_0 \iint\limits_{Q_T}
{h^{\alpha}h^2_{xx}\,dx dt} + a_0 \tfrac{\alpha(1 - \alpha)}{6}
\iint\limits_{Q_T} {h^{\alpha -2}h^4_{x}\,dx dt}\\
& \hspace{.2in} \label{E:b9a} \leq \int\limits_{\Omega}
{G^{(\alpha)}_{\varepsilon}(h_{0\varepsilon})\,dx} + d_1
\int\limits_0^T \max\left\{ 1, \left( \int\limits_\Omega h_x^2 \; dx
\right)^{\tfrac{\alpha}{2}+1} \right\}dt
\end{align}
where
$$
d_1 = b_4\left(\tfrac{3 a_1^2}{2 a_0 \alpha (1-\alpha)} +
\tfrac{|a_2| \| w'' \|_\infty}{1+\alpha}\right) + b_5 \; \left(
 \tfrac{3 a_1^2}{2 a_0 \alpha (1-\alpha)} \; M_\eps^{\alpha+2}
 +  \tfrac{|a_2| \| w'' \|_\infty}{1+\alpha} \;
 M_\eps^{\alpha+1}\right).
$$
Using the Cauchy inequality in (\ref{D:aa3}) and taking $\delta
\to 0$ yields
\begin{align} \label{E:b9b}
& \int\limits_{\Omega} {h_x^2 \,dx} + a_0
\iint\limits_{Q_T} {f_{\varepsilon}(h) h^2_{xxx}\,dx dt} \\
& \notag \hspace{.2in} \leq \int\limits_{\Omega} {h_{0\varepsilon,
x}^2\,dx} + \tfrac{2 a_1^2}{a_0} \iint\limits_{Q_T} { h^3 h^2_{x}
\,dx dt} + \tfrac{2 a_2^2 \| w' \|^2_\infty}{a_0} \iint\limits_{Q_T}
{ h^3 \,dx dt}.
\end{align}
Applying the Cauchy-Schwartz inequality and (\ref{LpH1}) yields
\begin{align} \notag
& \int\limits_{\Omega} {h_x^2 \,dx} + a_0 \iint\limits_{Q_T}
{f_{\varepsilon}(h) h^2_{xxx}\,dx dt}
\leq  \int\limits_{\Omega} {h_{0\varepsilon, x}^2\,dx} \\
& \hspace{.2in} \notag + \tfrac{a_0 \alpha(1-\alpha)}{6}
\iint\limits_{Q_T} { h^{\alpha-2} h^4_{x} \,dx dt} + d_2
\int\limits_0^T \max\left\{ 1, \left( \int\limits_\Omega h_x^2 \; dx
\right)^{4-\tfrac{\alpha}{2}} \right\} \; dt
\end{align}
where
$$
d_2 = b_4\left(\tfrac{6 a_1^4}{ a_0^3 \alpha(1-\alpha)}+ \tfrac{2
a_2^2}{a_0} \| w' \|_\infty^2\right) + b_5 \left( \tfrac{6
a_1^4}{a_0^3 \alpha(1-\alpha)} \; M_\eps^{8-\alpha} + \tfrac{2
a_2^2}{a_0} \| w' \|_\infty^2  \; M_\eps^3 \right).
$$
Adding $\int G_\eps^{(\alpha)}(h(x,T))$ to both sides of
(\ref{E:b9b}), $a_0 \iint h^\alpha h_{xx}^2$ to the resulting
righthand side, and using (\ref{E:b9a}),
\begin{align} \label{E:b9e}
& \int\limits_{\Omega} {h_x^2(x,T) \,dx}
+ \int\limits_{\Omega} {G^{(\alpha)}_{\varepsilon}(h(x,T))\,dx}
+ a_0
\iint\limits_{Q_T} {f_{\varepsilon}(h) h^2_{xxx}\,dx dt} \\
& \hspace{.2in} \notag \leq \notag
 \int\limits_{\Omega} {h_{0\varepsilon, x}^2\,dx}
+ \int\limits_\Omega G^{(\alpha)}_{\varepsilon}(h_{0\varepsilon})
\,dx + d_3 \int\limits_0^T \max\left\{ 1, \left( \int\limits_\Omega
h_x^2 \; dx \right)^{4-\tfrac{\alpha}{2}} \right\}
\end{align}
where $d_3 = d_1 + d_2$. Applying the nonlinear Gr\"onwall lemma
\cite{Bihari} to
$$
v(T) \leq v(0) + d_3 \int\limits_0^T \max\{1,v^{4-\alpha/2}(t)\}
\; dt$$ with $v(T) = \int (h_x^2(x,T) + \:
G_{\eps}^{(\alpha)}(h(x,T))) \; dx$ yields,
\begin{align} \label{bound1}
& \int\limits_\Omega (h_x^2(x,T) + G_{\eps}^{(\alpha)}(h(x,T))) \; dx \\
& \hspace{.4in} \leq 4^{\tfrac{1}{6-\alpha}} \max\left\{ 1,
\int\limits_\Omega ({h_{0,\eps}}_x^2(x) +
G_{\eps}^{(\alpha)}(h_{0,\eps}(x))) \; dx \right\} = K_\eps < \infty
\notag
\end{align}
for all $T$:
$$
0 \leq T \leq T_{\eps,loc}^{(\alpha)} := \tfrac{1}{d_3(6-\alpha)}
\min \Bigl\{ 1,   \Bigl(  \int\limits_\Omega ({h_{0,\eps}}_x^2(x)
+ G_{\eps}^{(\alpha)}(h_{0,\eps}(x))) \; dx
\Bigr)^{-\tfrac{6-\alpha}{2}} \Bigr\}.
$$
The bound (\ref{bound1}) holds for all $0 < \eps < \eps_0$ where
$\eps_0$ is from Lemma \ref{MainAE} and for all $t \leq
\min\{T_{loc},T_{\eps,loc}^{(\alpha)}\}$ where $T_{loc}$ is from Lemma
\ref{MainAE}.

Using the $\eps \to 0$ convergence of the initial data and the
choice of $\theta \in (0,2/5)$ (see (\ref{D:inreg})) as well as
the assumption that the initial data $h_0$ has finite
$\alpha$-entropy (\ref{finite_alpha_ent}), the times $T_{\eps,
loc}^{(\alpha)}$ converge to a positive limit and the upper bound
$K_\eps$ in (\ref{bound1}) can be taken finite and independent of
$\varepsilon$. (We refer the reader to the end of the proof of
Lemma \ref{F:local_BF} in this Appendix for a fuller explanation
of a similar case.) Therefore there exists $\eps_0^{(\alpha)}$ and
$K$ such that the bound (\ref{bound1}) holds for all $0 < \eps <
\eps_0^{(\alpha)}$ with $K$ replacing $K_\eps$ and for all
\begin{equation} \label{Tloc_alpha_is}
0 \leq t \leq T_{loc}^{(\alpha)} := \min\left\{T_{loc},
\tfrac{9}{10} \lim_{\eps \to 0} T_{\eps,loc}^{(\alpha)}
\right\}
\end{equation}
where $T_{loc}$ is the time from Lemma \ref{MainAE}.  Also, without
loss of generality, $\eps_0^{(\alpha)}$ can be taken to be less than
or equal to the $\eps_0$ from Lemma \ref{MainAE}.

Using the uniform bound on $\int h_x^2$ that (\ref{bound1})
provides, one can find a uniform-in-$\eps$ bound for the
right-hand-side of (\ref{E:b9a}) yielding the desired bound
\begin{equation} \label{alpha_bound1}
\int\limits_{\Omega} {G^{(\alpha)}_{\varepsilon}(h(x,T))\,dx} + a_0
\iint\limits_{Q_T} {h^{\alpha}h^2_{xx}}\,dx dt
 +
a_0 \tfrac{\alpha(1 - \alpha)}{6} \iint\limits_{Q_T} {h^{\alpha -2
}h^4_{x}\,dx dt}  \leq K_1
\end{equation}
which holds for all $0 <\eps < \eps_0^{(\alpha)}$ and all
$0 \leq T \leq T_{loc}^{(\alpha)}$.

It remains to argue that (\ref{alpha_bound1}) implies that
for all $0 < \eps < \eps_0^{(\alpha)}$ that
$h_\eps^{\alpha/2 + 1}$ and $h_\eps^{\alpha/4 + 1/2}$ are contained
in
balls in $L^{2}(0, T; H^2(\Omega))$
and $L^{2}(0, T; W^1_4(\Omega))$ respectively.  It suffices to show that
$$
\iint\limits_{Q_T}
\left(
h_\eps^{\alpha/2 + 1}\right)^2_{xx} \; dx dt \leq K, \qquad
\iint\limits_{Q_T}
\left(
h_\eps^{\alpha/4 + 1/2}\right)^4_x \; dx dt \leq K
$$
for some $K$ that is independent of $\eps$ and $T$.
The integral $\iint (h_\eps^{\alpha/2+1})_{xx}^2$ is a linear
combination of
$\iint h^{\alpha-2} h_x^4$,
$\iint h^{\alpha-1} h_x^2 h_{xx}$, and
$\iint h^{\alpha} h_{xx}^2$.
Integration by parts and the periodic boundary conditions imply
\begin{equation} \label{calc_ident}
\tfrac{1 - \alpha}{3} \iint\limits_{Q_T} {h^{\alpha -2
}h^4_{x}\,dx dt} = \iint\limits_{Q_T} {h^{\alpha -1 }h^2_{x}
h_{xx}\,dx dt}
\end{equation}
Hence $\iint (h_\eps^{\alpha/2+1})_{xx}^2$ is a linear combination of
$\iint h^{\alpha-2} h_x^4$, and $\iint h^{\alpha} h_{xx}^2$.  By
(\ref{alpha_bound1}), the two integrals are uniformly bounded independent
of $\eps$ and $T$ hence $\iint (h_\eps^{\alpha/2+1})_{xx}^2$ is as
well, yielding the first part of (\ref{E:b13}).

The uniform bound of $\iint (h_\eps^{\alpha/4 + 1/2})_x^4$ follows
immediately from the uniform bound of $\iint h^{\alpha-2} h_x^4$,
yielding the second part of (\ref{E:b13}).

\textbf{Case 2: $-\tfrac{1}{2} < \alpha < 0$.}  For $\alpha < 0$ the
coefficient multiplying $\iint h^{\alpha-2} h_x^4$ in (\ref{E:b1})
is negative.  However, we will show that if $\alpha > -1/2$ then one
can replace this coefficient with a positive coefficient while also
controlling the term $\iint h^\alpha h_x^2$ on the right-hand side
of (\ref{E:b1}).

Applying the Cauchy-Schwartz inequality to the right--hand side of
(\ref{calc_ident}), dividing by $\sqrt{\iint h^{\alpha-2} h_x^4}$,
and squaring both sides of the resulting inequality yields
\begin{equation}\label{E:b4}
\iint\limits_{Q_T} {h^{\alpha -2 }h^4_{x}\,dx dt} \leq \tfrac{9}{(1
- \alpha)^2} \iint\limits_{Q_T} {h^{\alpha} h^2_{xx}\,dx dt} \qquad
\forall \alpha < 1.
\end{equation}
Using (\ref{E:b4}) in (\ref{E:b1}) yields
\begin{align}
& \label{E:b5}
\int\limits_{\Omega} {G^{(\alpha)}_{\varepsilon}(h(x,T))\,dx} + a_0
\tfrac{1+2\alpha}{1-\alpha}\iint\limits_{Q_T} {h^{\alpha}h^2_{xx}\,dx dt}  \\
& \hspace{.2in} \notag
\leq
\int\limits_{\Omega}
{G^{(\alpha)}_{\varepsilon}(h_{0\varepsilon})\,dx} + a_1
\iint\limits_{Q_T} {h^{\alpha} h^2_{x}\,dx dt}
+ \tfrac{|a_2|}{\alpha+1} \| w'' \|_\infty  \iint\limits_{Q_T} {h^{\alpha+1}  \,dx dt}.
\notag
\end{align}
Note that if $\alpha > -1/2$ then all the terms on the left--hand side
of (\ref{E:b5}) are positive.  We now control the term
$\iint h^\alpha h_x^2$ on the right-hand side of
(\ref{E:b5}).

By integration by parts and the periodic boundary conditions
\begin{equation}  \label{calc_ident2}
\iint\limits_{Q_T} h^\alpha h_x^2 \; dx dt = - \tfrac{1}{1+\alpha}
\iint\limits_{Q_T} h^{\alpha+1} h_{xx} \; dx dt.
\end{equation}
Applying the Cauchy inequality to (\ref{calc_ident2}) yields
\begin{equation} \label{E:b6}
a_1 \iint\limits_{Q_T} h^\alpha h_x^2\; dx dt \leq
\iint\limits_{Q_T}\left(\tfrac{a_0 (1+2 \alpha)}{2(1-\alpha)}
h^\alpha h_{xx}^2 + \tfrac{a_1^2 (1-\alpha)}{2 a_0 (1+2 \alpha)
(1+\alpha)^2} h^{\alpha+2}\right) dxdt.
\end{equation}
Using inequality (\ref{E:b6}) in (\ref{E:b5}) yields
\begin{align}
& \label{E:b6a}
\int\limits_{\Omega} {G^{(\alpha)}_{\varepsilon}(h(x,T))\,dx} + a_0
\tfrac{1+2\alpha}{2(1-\alpha)}\iint\limits_{Q_T} {h^{\alpha}h^2_{xx}\,dx dt}  \\
& \notag \leq \int\limits_{\Omega}
{G^{(\alpha)}_{\varepsilon}(h_{0\varepsilon})\,dx} +
\iint\limits_{Q_T} \left(\tfrac{a_1^2 (1-\alpha)}{2 a_0 (1+2 \alpha)
(1+\alpha)^2} h^{\alpha+2}+ \tfrac{|a_2|}{\alpha+1} \| w'' \|_\infty
{h^{\alpha+1} }\right) dxdt. \notag
\end{align}
Adding
$$
\tfrac{a_0 (1+2 \alpha) (1-\alpha)}{36}  \iint\limits_{Q_T}  h^{\alpha-2} h_x^4 \; dx dt
$$
to both sides of (\ref{E:b6a}) and using the inequality (\ref{E:b4}) yields
\begin{align}
& \label{E:b6b}
\int\limits_{\Omega} {G^{(\alpha)}_{\varepsilon}(h(x,T))\,dx} + a_0
\tfrac{(1+2\alpha)}{4(1-\alpha)}\iint\limits_{Q_T} {h^{\alpha}h^2_{xx}\,dx dt}  \\
& \hspace{.2in} \notag + \tfrac{a_0(1+2\alpha)(1-\alpha)}{36}
\iint\limits_{Q_T} h^{\alpha-2} h_x^4 \; dx dt \leq
\int\limits_{\Omega}
{G^{(\alpha)}_{\varepsilon}(h_{0\varepsilon})\,dx} \\
& \hspace{.2in} \notag + \tfrac{a_1^2 (1-\alpha)}{2 a_0 (1+2 \alpha)
(1+\alpha)^2} \iint\limits_{Q_T} h^{\alpha+2}\; dx dt +
\tfrac{|a_2|}{\alpha+1} \| w'' \|_\infty \iint\limits_{Q_T}
{h^{\alpha+1} \; dx dt}. \notag
\end{align}
Using (\ref{E:b6b}) and (\ref{LpH1}) yields
\begin{align}
&\notag \int\limits_{\Omega}
{G^{(\alpha)}_{\varepsilon}(h(x,T))\,dx} + \iint\limits_{Q_T}
\left(\tfrac{a_0(1+2\alpha)}{4(1-\alpha)} {h^{\alpha}h^2_{xx}} +
\tfrac{a_0(1+2\alpha)(1-\alpha)}{36}h^{\alpha -2
}h^4_{x}\right)\,dxdt \\
& \hspace{.2in} \label{E:b9a'} \leq \int\limits_{\Omega}
{G^{(\alpha)}_{\varepsilon}(h_{0\varepsilon})\,dx} + e_1
\int\limits_0^T \max\left\{ 1, \left( \int\limits_\Omega h_x^2 \; dx
\right)^{\tfrac{\alpha}{2}+1} \right\} \; dt
\end{align}
where $ e_1 =
b_4\left(\tfrac{a_1^2(1-\alpha)}{2a_0(1+2\alpha)(1+\alpha)^2}+
\tfrac{|a_2|}{\alpha+1} \|w''\|_\infty \right) + b_5 \bigl(
\tfrac{a_1^2(1-\alpha)}{2a_0(1+2\alpha)(1+\alpha)^2}\;
M_\eps^{\alpha+2}$ $ + \tfrac{|a_2|}{\alpha+1} \|w''\|_\infty \;
M_\eps^{\alpha+1} \bigr)$. Recall the bound (\ref{E:b9b}). As
before, by the Cauchy inequality,
\begin{align} \label{E:b9c'}
& \tfrac{2 a_1^2}{a_0} \iint\limits_{Q_T} h^3 h_x^2 \; dx dt \leq
\tfrac{a_0(1+2\alpha)(1-\alpha)}{36}
\iint\limits_{Q_T} h^{\alpha-2} h_x^4 \; dx dt\\
& \hspace{1.4in}  \notag + \tfrac{36 a_1^4}{a_0^3
(1+2\alpha)(1-\alpha)} \iint\limits_{Q_T} h^{8-\alpha} \; dx dt.
\end{align}
Using (\ref{E:b9c'}) in (\ref{E:b9b}) yields
\begin{align} \notag
& \int\limits_{\Omega} {h_x^2 \,dx} + a_0 \iint\limits_{Q_T}
{f_{\varepsilon}(h) h^2_{xxx}\,dx dt}\leq  \int\limits_{\Omega}
{h_{0\varepsilon, x}^2\,dx}\\
& \notag \hspace{.2in}  + \tfrac{a_0 (1+2\alpha)(1-\alpha)}{36}
\iint\limits_{Q_T} { h^{\alpha-2} h^4_{x} \,dx dt} + e_2
\int\limits_0^T \max\Bigl\{ 1, \Bigl( \int\limits_\Omega h_x^2 \; dx
\Bigr)^{4-\tfrac{\alpha}{2}} \Bigr\} \; dt
\end{align}
where $ e_2 = b_4\left(\tfrac{36 a_1^4}{a_0^3 (1+2\alpha)(1-\alpha)}
+ \tfrac{2 a_2^2}{a_0} \| w' \|^2_\infty \right) + b_5 \bigl(
\tfrac{36 a_1^4}{a_0^3 (1+2\alpha)(1-\alpha)} \; M_\eps^{8-\alpha}
+$ $ \tfrac{2 a_2^2}{a_0} \| w' \|^2_\infty \; M_\eps^{3} \bigr)$.
Just as (\ref{E:b9a}) and (\ref{E:b9b}) yielded (\ref{E:b9e}),
(\ref{E:b9a'}) combined with the above inequality yields
\begin{align} \label{E:b9e'}
& \int\limits_{\Omega} {h_x^2(x,T) \,dx}
+ \int\limits_{\Omega} {G^{(\alpha)}_{\varepsilon}(h(x,T))\,dx}
+ a_0
\iint\limits_{Q_T} {f_{\varepsilon}(h) h^2_{xxx}\,dx dt} \\
& \hspace{.2in} \leq \notag
 \int\limits_{\Omega} {h_{0\varepsilon, x}^2\,dx}
+ \int\limits_\Omega G^{(\alpha)}_{\varepsilon}(h_{0\varepsilon})
\,dx + e_3 \int\limits_0^T \max\Bigl\{ 1, \Bigl(
\int\limits_\Omega h_x^2 \; dx \Bigr)^{4-\tfrac{\alpha}{2}}
\Bigr\}
\end{align}
where $e_3 = e_1 + e_2$. The rest of the proof now continues as in
the $0 < \alpha < 1$ case. Specifically, one finds a bound
\begin{align} \label{bound2}
& \int\limits_\Omega (h_x^2(x,T) + G_{\eps}^{(\alpha)}(h(x,T))) \; dx \\
& \hspace{.4in} \leq 4^{\tfrac{1}{6-\alpha}} \max\left\{ 1,
\int\limits_\Omega ({h_{0,\eps,x}}^2(x) +
G_{\eps}^{(\alpha)}(h_{0,\eps}(x))) \; dx \right\} = K_\eps < \infty
\notag
\end{align}
for all $T$:
$$
0 \leq T \leq T_{\eps,loc}^{(\alpha)} := \tfrac{1}{e_3(6-\alpha)}
\min\Bigl\{ 1,   \Bigl(  \int\limits_\Omega ({h_{0,\eps,x}}^2(x) +
G_{\eps}^{(\alpha)}(h_{0,\eps}(x))) \;
dx\Bigr)^{-\frac{6-\alpha}{2}} \Bigr\}.
$$
The time $T_{loc}^{(\alpha)}$ is defined as in (\ref{Tloc_alpha_is})
and the uniform bound (\ref{bound2}) used to bound the right hand side
of (\ref{E:b9a'}) yields the desired bound
\begin{align} \notag
& \int\limits_{\Omega} {G^{(\alpha)}_{\varepsilon}(h(x,T))\,dx}
+ \tfrac{a_0(1+2\alpha)}{4(1-\alpha)}
\iint\limits_{Q_T} {h^{\alpha}h^2_{xx}\,dx dt}  \\
& \hspace{1.2in} + \tfrac{a_0(1+2\alpha)(1-\alpha)}{36}
\iint\limits_{Q_T} {h^{\alpha -2 }h^4_{x}\,dx dt}  \leq K_2.
\label{alpha_bound2}
\end{align}

\end{proof}

\begin{proof}[Proof of Lemma~\ref{F:local_BF}]
In the following, we denote the
 positive, classical solution $h_\eps$
constructed in Lemma \ref{preC:Th1} by $h$ (whenever there is no
chance of confusion).

Recall the entropy function $G_{\delta \eps}(z)$ defined by (\ref{D:reg2}).
Multiplying
(\ref{D:1r'}) by $\xi(x) G'_{\delta \varepsilon}(h_{\delta \eps})$,
taking $\delta \to 0$, and integrating over $Q_T$ yields
\begin{align}
& \notag \int\limits_{\Omega} {\xi(x) G_{\varepsilon}(h(x,T))dx}  -
\int\limits_{\Omega} {\xi(x) G_{\varepsilon}(h_{0,\varepsilon})dx} =
- a_3\iint\limits_{Q_T} {\xi(x) G'_{\varepsilon}(h)
h_{x}dx dt} \\
& \notag \hspace{.3in}
+ \iint\limits_{Q_T} {f_{\varepsilon}(h)(a_0 h_{xxx} +
a_1 h_x + a_2 w') (\xi' G'_{\varepsilon}(h) + \xi G''_{\varepsilon}(h) h_x) \,dx dt}  \\
&\notag =
a_3\iint\limits_{Q_T} {\xi'G_{\varepsilon}(h) \,dx dt} +
\iint\limits_{Q_T} {\xi' f_{\varepsilon}(h)G'_{\varepsilon}(h)
(a_0 h_{xxx} +
a_1 h_x + a_2 w')\,dx dt}  \\
&  \hspace{.3in} +
\iint\limits_{Q_T} {\xi  h_x (a_0 h_{xxx} + a_1 h_x + a_2 w')
\,dx dt} =: I_{1} + I_{2} + I_{3} .
\label{F:pp1}
\end{align}
One easily finds that for all $\eps > 0$ and all $z \geq 0$
$$
|f_{\varepsilon}(z)G_{\varepsilon}'(z)| \leqslant \tfrac{1}{2}z, \
|f_{\varepsilon}'(z)G_{\varepsilon}'(z)| \leqslant 2,
$$
$$
\Bigl |\int \limits_0^z
{f_{\varepsilon}(s)G'_{\varepsilon}(s)\,ds}\Bigr| \leq
\tfrac{1}{2} z^2 + \tfrac{3}{5} \text{ if } 0 < \eps <
(\sqrt{33}-3)/4.
$$
Using these bounds, and recalling $\xi = \zeta^4$, we bound
$|I_2|$:
\begin{align}
& \notag | I_2 | \leq \iint\limits_{Q_T} \left(\tfrac{a_0}{2}\zeta^4
h_{xx}^2 + \gamma_1 \left[ \zeta^2 \zeta_x^2 + \zeta^3 | \zeta_{xx}|
+ \zeta_x^4  + \zeta^2 \zeta_{xx}^2 \right] \left( h^2 + h_x^2 \right)\right)\,dxdt \\
& \label{I2_bound} \hspace{.4in} + 2 |a_2| \| w' \|_\infty
\iint\limits_{Q_T} \zeta^3 | \zeta_x | h \; dx dt+ \tfrac{3}{5}
|a_1| \iint\limits_{Q_T} |  \xi^{\prime \prime} |\; dx dt
\end{align}
where $\gamma_1 = \max\{ 102 a_0, 6 |a_1| \}$ and $0 < \eps  <
(\sqrt{33}-3)/4$. Now, integrating by parts in $I_3$, we deduce
\begin{align}
& \notag I_3 + a_0 \iint\limits_{Q_T} \xi h_{xx}^2 \; dx dt
\leq \gamma_2 \iint\limits_{Q_T} \left[ \zeta^2 \zeta_x^2 + \zeta^3 | \zeta_{xx} | + \zeta^4 \right] h_x^2 \; dx dt\\
 & \hspace{.7in} \label{I3_bound}
 + 4 |a_2| \left( \| w' \|_\infty + \| w^{\prime \prime} \|_\infty \right) \iint\limits_{Q_T} \left( \zeta^3 | \zeta_x | + \zeta^4 \right) h\; dx dt
\end{align}
where $\gamma_2 = \max\{ 6 a_0, |a_1| \}$. Using bounds
(\ref{I2_bound}) and (\ref{I3_bound}) we obtain that
\begin{equation} \label{local_BF}
\int\limits_{\Omega} {\xi G_{\varepsilon}(h_\eps(x,T))\,dx}
\leqslant \int\limits_{\Omega} {\xi G_{\varepsilon}(h_{0 \eps})\,dx}
+ C
\end{equation}
where $C > 0$ is independent of $\varepsilon > 0$. Using the fact
that $\theta$ was chosen so that $\theta < 2/5 < 1/2$, we have
$|\xi(x) \, G_\eps(h_{0 \eps}(x))| \leq \xi(x) (G_0(h_0(x)) + c)
\leq C (G_0(h_0(x)) + c )$ almost everywhere in $x$ and for all
$\eps < \eps_0$. To finish the proof we apply Fatou's lemma to the
left-hand side and Lebesgue Dominated Convergence Theorem to the
right-hand side  of (\ref{local_BF}).
\end{proof}

\section{Results used from functional analysis}

\begin{lemma}\label{A.1}
(\cite{Lions}) Suppose that $X,\ Y,$  and  $Z$ are Banach spaces,
$X \!\Subset \!Y \subset \! Z$, and $X$ and $Z$ are reflexive.
Then the embedding $ \{ u \in \! L^{p_0 } (0,T;$ $X):$ $\partial
_t u \in L^{p_1 } (0,T;Z),1 < p_i < \infty ,i = 0,1 \} \Subset
L^{p_0 } (0,T;Y)$ is compact.
\end{lemma}

\begin{lemma}\label{A.2}
(\cite{Sim}) Suppose that $X,\ Y,$  and  $Z$ are Banach spaces and
$X \!\Subset \!Y\hspace{-0.2cm} \subset \! Z$. Then the embedding
$ \{u \in L^\infty (0,T;X):\partial _t u \in L^p (0,T;Z)$, $p > 1
\} \Subset C(0,T;Y)$ is compact.
\end{lemma}

\bibliographystyle{plain}
\bibliography{ThinFilm}

\end{document}